\documentclass[10pt,a4paper]{article}

\usepackage{amsmath,amsthm,amsfonts,amssymb,mathtools}
\usepackage{enumitem}\setitemize{leftmargin=5mm}
\usepackage[dvipsnames,svgnames,table]{xcolor}
\usepackage{soul}
\usepackage{pgf,tikz}
\usetikzlibrary{matrix, calc, arrows,shapes,decorations}
\usepackage{pgfplots} 
\usepgfplotslibrary{colorbrewer,groupplots}
\pgfplotsset{
	discard if not/.style 2 args={
	x filter/.append code={
	\edef\tempa{\thisrow{#1}}
	\edef\tempb{#2}
	\ifx\tempa\tempb
	\else
	
	\fi
	}
	}
}
\usepackage{pgfplotstable} 
\usepackage{booktabs} 
\usepackage{colortbl}
\makeatletter
\pgfplotstableset{
	discard if not/.style 2 args={
	row predicate/.append code={
	\def\pgfplotstable@loc@TMPd{\pgfplotstablegetelem{##1}{#1}\of}
	\expandafter\pgfplotstable@loc@TMPd\pgfplotstablename
	\edef\tempa{\pgfplotsretval}
	\edef\tempb{#2}
	\ifx\tempa\tempb
	\else
	\pgfplotstableuserowfalse
	\fi
	}
	}
}
\makeatother

\pgfplotscreateplotcyclelist{paulcolors}{%
	violet,every mark/.append style={solid,fill=violet},mark=square*,very thick,mark size=3pt\\%
	violet,densely dashed,every mark/.append style={solid,fill=violet},mark=square*,very thick,mark size=3pt\\%
	teal,every mark/.append style={solid,fill=teal},mark=*,very thick,mark size=3pt\\%
	teal,densely dashed,every mark/.append style={solid,fill=teal},mark=*,very thick,mark size=3pt\\%
	orange,every mark/.append style={solid,fill=orange},mark=diamond*,very thick,mark size=3pt\\%
	orange,densely dashed,every mark/.append style={solid,fill=orange},mark=diamond*,very thick,mark size=3pt\\%
	cyan,every mark/.append style={solid,fill=cyan},mark=star,very thick,mark size=3pt\\%
	cyan,densely dashed,every mark/.append style={solid,fill=cyan},mark=star,very thick,mark size=3pt\\%
}

\pgfplotscreateplotcyclelist{paulcolorsCOND}{%
	teal,every mark/.append style={solid,fill=teal},mark=*,very thick,mark size=3pt\\%
	teal,every mark/.append style={solid,fill=teal},mark=*,very thick,mark size=3pt\\%
	teal,densely dashed,every mark/.append style={solid,fill=teal},mark=*,very thick,mark size=3pt\\%
	teal,densely dashed,every mark/.append style={solid,fill=teal},mark=*,very thick,mark size=3pt\\%
	orange,every mark/.append style={solid,fill=orange},mark=diamond*,very thick,mark size=3pt\\%
	orange,every mark/.append style={solid,fill=orange},mark=diamond*,very thick,mark size=3pt\\%
	orange,densely dashed,every mark/.append style={solid,fill=orange},mark=diamond*,very thick,mark size=3pt\\%
	orange,densely dashed,every mark/.append style={solid,fill=orange},mark=diamond*,very thick,mark size=3pt\\%
	red,every mark/.append style={solid,fill=red},mark=star,very thick,mark size=3pt\\%
	red,every mark/.append style={solid,fill=red},mark=star,very thick,mark size=3pt\\%
	red,densely dashed,every mark/.append style={solid,fill=red},mark=star,very thick,mark size=3pt\\%
	red,densely dashed,every mark/.append style={solid,fill=red},mark=star,very thick,mark size=3pt\\%
}

\pgfplotsset{compat=1.14}

\allowdisplaybreaks[4]

\usepackage[lined,boxed,commentsnumbered]{algorithm2e}

\usepackage{caption}
\SetAlCapSkip{1em}
  
\SetAlgoSkip{SkipBeforeAndAfter}

\newcommand*{\N}[1]{\left\|#1\right\|}
\newcommand*{\abs}[1]{\left|#1\right|}
\newcommand*{\jmp}[1]{[\![#1]\!]}    
\newcommand*{\mvl}[1]{\{\!\!\{#1\}\!\!\}}

\newcommand{\IN}{\mathbb{N}}\newcommand{\IR}{\mathbb{R}}
\newcommand{\IP}{\mathbb{P}}\newcommand{\IT}{\mathbb{T}}
\newcommand{\QT}{{\mathbb{Q\!T}}}
\newcommand{\calA}{{\mathcal A}}
\newcommand{\calB}{{\mathcal B}}

\newcommand{\calF}{{\mathcal F}}
\newcommand{\calL}{{\mathcal L}}
\newcommand{\calO}{{\mathcal O}}
\newcommand{\calT}{{\mathcal T}}
\newcommand{\calP}{{\mathcal P}}
\newcommand{\calM}{{\mathcal M}}

\newcommand{\Uu}[1]{{\bm{#1}}}

\newcommand{\be}{{\Uu e}}
\newcommand{\bi}{{\Uu i}}
\newcommand{\bj}{{\Uu j}}
\newcommand{\bk}{{\Uu k}}
\newcommand{\bl}{{\Uu\ell}}
\newcommand{\bn}{{\Uu n}}

\newcommand{\br}{{\Uu r}}
\newcommand{\bw}{{\Uu w}}
\newcommand{\bx}{{\Uu x}}
\newcommand{\by}{{\Uu y}}

\newcommand{\bzero}{\Uu{0}}
\newcommand{\bK}{{\Uu K}}
\newcommand{\bI}{{\Uu I}}
\newcommand{\bbeta}{{\Uu\beta}}
\newcommand{\bxi}{{\Uu\xi}}

\newcommand{\tayt}{{\mathsf{T}}}
\newcommand{\Csr}{C_{\mathrm{sr}}}
\newcommand{\Cg}{{C_{\mathrm g}}}
\newcommand{\kmin}{{k_{\min}}}

\newcommand{\vertiii}[1]{{\left\vert\kern-0.25ex\left\vert\kern-0.25ex\left\vert #1 
	\right\vert\kern-0.25ex\right\vert\kern-0.25ex\right\vert}}

\usepackage[english]{babel}
\usepackage[utf8]{inputenc}
\usepackage[T1]{fontenc}
\usepackage{bm}
\usepackage{csquotes}

\usetikzlibrary{decorations.markings}
\tikzset{middlearrow/.style={
	decoration={markings,
	mark= at position 0.5 with {\arrow{#1}} ,
	},
	postaction={decorate}
	}
}

\usepackage[backend=biber,
style=numeric,
sorting=nyt,
maxbibnames=6,
giveninits=true,
isbn=false,url=false,doi=false]{biblatex}
\usepackage[left=3cm,top = 3cm, bottom = 3cm,right= 3cm,]{geometry} 

\usepackage[bookmarks=true, bookmarksopen=true, bookmarksopenlevel=5]{hyperref} 
\definecolor{myblue}{rgb}{0,0,0.6}     
\hypersetup{pdftitle={Polynomial quasi-Trefftz DG for PDEs with smooth coefficients: elliptic problems},  
	colorlinks=true, linkcolor=myblue,  citecolor=myblue, filecolor=myblue,   urlcolor=myblue,  }
\usepackage{cleveref} 

\newtheorem{theorem}{Theorem}[section] 
\newtheorem{lemma}[theorem]{Lemma}
\newtheorem{prop}[theorem]{Proposition}
\newtheorem{rem}[theorem]{Remark}
\newtheorem{Def}[theorem]{Definition}

\usepackage{tabularray}

\pgfplotsset{
	discard if/.style 2 args={
		x filter/.append code={
			\edef\tempa{\thisrow{#1}}
			\edef\tempb{#2}
			\ifx\tempa\tempb
			
			\fi
		}
	},}

\begin{document}

\title{
Polynomial quasi-Trefftz DG for PDEs with smooth coefficients: elliptic problems
}
\author{Lise-Marie Imbert-G\'erard\thanks{Department of Mathematics, University of Arizona, USA  (\href{mailto:lmig@arizona.edu}{lmig@arizona.edu})} ,
Andrea Moiola\thanks{Department of Mathematics, University of Pavia, Italy       (\href{mailto:andrea.moiola@unipv.it}{andrea.moiola@unipv.it})} ,
Chiara Perinati\thanks{Department of Mathematics, University of Pavia, Italy (\href{mailto:chiara.perinati01@universitadipavia.it}{chiara.perinati01@universitadipavia.it})} , 
Paul Stocker\thanks{Faculty of Mathematics, University of Vienna, Austria (\href{mailto:paul.stocker@univie.ac.at}{paul.stocker@univie.ac.at})}}
\date{\today}

\maketitle

\begin{abstract}
Trefftz schemes are high-order Galerkin methods whose discrete spaces are made of elementwise exact solutions of the underlying PDE. 
Trefftz basis functions can be easily computed for many PDEs that are linear, homogeneous, and have piecewise-constant coefficients. 
However, if the equation has variable coefficients, exact solutions are generally unavailable. 
Quasi-Trefftz methods overcome this limitation relying on elementwise ``approximate solutions'' of the PDE, in the sense of Taylor polynomials.
	
We define polynomial quasi-Trefftz spaces for general linear PDEs with smooth coefficients and source term,
describe their approximation properties and, under a non-degeneracy condition, provide a simple algorithm to compute a basis. 
	We then focus on a quasi-Trefftz DG method for variable-coefficient elliptic diffusion--advection--reaction problems, showing stability and high-order convergence of the scheme. 
	The main advantage over standard DG schemes is the higher accuracy for comparable numbers of degrees of freedom.
	For non-homogeneous problems with piecewise-smooth source term we propose to construct a local quasi-Trefftz particular solution and then solve for the difference. 
	Numerical experiments in 2 and 3 space dimensions show the excellent properties of the method both in diffusion-dominated and advection-dominated problems.
	
	
	\end{abstract}
	
	\section{Introduction}\label{s:Intro}
	\subsection{Motivation for quasi-Trefftz methods}
	Classical Galerkin schemes, such as finite element and discontinuous Galerkin (DG) methods, seek an approximation of a boundary value problem (BVP) solution in a piecewise-polynomial discrete space.
	The most common trial and test spaces contain all piecewise polynomials of some given maximal degree, possibly with some inter-element continuity.
	These spaces are not tuned to approximate the solutions of a given partial differential equation (PDE), instead they contain approximations to all sufficiently regular functions.
	To reduce the number of degrees of freedom (DOFs), i.e.\ the size of the discrete space, and thus the computational cost of the scheme, one can construct more specialized discrete spaces, that are adapted to the PDE to be approximated.
	
	A well-known way to implement this idea is to use a \emph{Trefftz} method: a scheme where all discrete functions are elementwise solutions of the PDE.
	This is feasible when the PDE has piecewise-constant coefficients.
	For instance, Trefftz methods for the Laplace equation $\Delta u=0$ use harmonic polynomials as basis functions \cite{LiShu2006,HMPS14}, while those for the wave equation $\partial_t^2u-\Delta u=0$ use ``polynomial wave'' solutions in space--time \cite{moiola2018space}.
	For PDEs with a zero-order term, no polynomial solutions are available, thus Trefftz methods for the Helmholtz equation $\Delta u+k^2u=0$ typically use complex-exponential plane-wave bases \cite{hiptmair2016survey}.

	The common feature of these Trefftz schemes is that they offer the same accuracy as comparable methods based on full polynomial spaces, using much fewer DOFs.
    A sparsity comparison of a Trefftz DG scheme against other polytopal finite element methods, including Hybrid-DG, Hybrid High-Order, and Virtual Element Methods, has been performed in \cite{2405.16864}. 
    The Trefftz DG scheme is shown to achieve a reduction in complexity comparable to that of the other methods. 
    Furthermore, as the degrees of freedom of the Trefftz DG method are only associated to the mesh elements, the Trefftz DG method generalizes very efficiently to polytopal meshes.
    In \cite{LLS_NM_2024}, the Trefftz DG method is presented for the Stokes problem, and compared to other methods in this context.
	
	However, when the PDE has variable coefficients, the construction of local exact solutions is usually not possible.
	Instead, \emph{quasi-Trefftz} methods can be applied in this case.
	These rely on discrete spaces of functions that, on each element, are solution of the PDE ``up to a small residual'': for the PDE $\calM u=f$, in each mesh element $E$ with diameter $h_E$, every function $v_h$ in the discrete trial space satisfies $|\calM v_h-f|=\calO(h_E^q)$ in $E$ for some fixed exponent $q\in\IN$.
	Both Trefftz and quasi-Trefftz methods are usually formulated as DG schemes.
    Another approach that allows for variable coefficients and non-zero right-hand sides is the embedded Trefftz method \cite{LS_IJMNE_2023}, where the Trefftz basis functions are not explicitly constructed but embedded in a standard DG method.
	
	So far, the quasi-Trefftz idea has been used for oscillatory problems with smooth coefficients:
    in the time-harmonic regime using both polynomial and complex-exponential basis functions \cite{IGS2024,imbert2014generalized},
	in space--time using polynomials for the wave \cite{IGMS_MC_2021} and the Schr\"odinger \cite{GomezMoiola2024} equations.
	Complex-exponential quasi-Trefftz spaces for some homogeneous equations of order $m\geq 2$ were introduced in \cite{IG2021ampl,IGS2021roadmap}, while in \cite{IGS2024} complex-exponential and polynomial quasi-Trefftz spaces were studied for some homogeneous equations of order $m=2$.
	However, no general treatment of the corresponding quasi-Trefftz methods is available.
	
	\subsection{The contributions of this paper}
	The main goal of this paper is to introduce and analyze the degree-$p$ polynomial quasi-Trefftz space (denoted $\QT_f^p(E)$) for the general, order-$m$, linear, partial differential operator $\calM =\sum_{|\bj|\le m}\alpha_\bj D^\bj$.
	The main assumption is that the coefficients $\alpha_\bj$ and the source term $f$ are sufficiently smooth, namely $\alpha_\bj,f\in C^{p-m}(E)$ (where $E$ will then be a mesh element).
	We define the affine space $\QT_f^p(E)$ in \Cref{def:QT}, and prove in \Cref{th:Approx} that it approximates all smooth solutions of $\calM u=f$ with the same convergence rate, with respect to the domain size, compared to the full polynomial space $\IP^p(E)$ of the same degree.
	Under a simple non-degeneracy condition \eqref{hpcoeff}, in \cref{s:ConstructionQT} we provide a simple iterative algorithm to compute the monomial expansion of all quasi-Trefftz polynomials.
	These functions are uniquely determined by their ``Cauchy data'', i.e.\ the values of the first $m$ derivatives on a given hyperplane.
	Their computation requires the partial derivatives of the PDE coefficients $\alpha_\bj$ and right-hand side $f$ at a fixed point $\bx^E$.
	\Cref{Algo:general} thus allows to construct simple bases of $\QT^p_0(E)$ and to verify that the dimension of this space is indeed much smaller than $\mathrm{dim}(\IP^p(E))$, see~\eqref{eq:dimreduc}.
	
	In the following sections we study a quasi-Trefftz DG method for elliptic diffusion--advection--reaction problems. 
	We introduce a BVP in \cref{s:ProblemDAR}, a polytopal mesh in \cref{s:Mesh}, and a DG formulation in \cref{s:DGformulation}.
	We use the classical symmetric interior penalty for the diffusion term and upwind penalization for the advection term.
	To study the convergence of this method, in \cref{s:WellP} we slightly modify the standard DG analysis of e.g.~\cite{di2011mathematical} to handle more general polynomial discrete spaces.
	This allows to prove optimal convergence rates for the quasi-Trefftz DG method in \cref{s:QTDGdiscretizazion}.
	Since the PDE source term $f$ enters the definition of $\QT^p_f(E)$, the trial space is actually an affine space: to write the Galerkin problem as a linear system, we compute an elementwise approximate PDE solution, using again \Cref{Algo:general}, possibly in parallel, and then solve for the difference, see \eqref{qtvariational2}.
	The quasi-Trefftz space could be combined with any other stable and quasi-optimal DG formulation with similar results.
	
	Finally, in \cref{s:numexp} we show some numerical examples in 2 and 3 space dimensions illustrating the capabilities of the method.
	In these examples, the method based on the same DG formulation, discretized with a polynomial quasi-Trefftz discrete space of degree $p$ compared to the full polynomial space of the same degree, achieves the same error and convergence rates, but with considerably fewer DOFs.
	The construction of the quasi-Trefftz basis, following \Cref{Algo:general}, involves a small overhead, but is completely parallelizable: \Cref{tab:ex2} shows that the total computational time for the quasi-Trefftz version of the scheme is lower than for the full-polynomial space.
	We also consider two advection-dominated examples and show that the solutions are well captured in both cases.
	
	The quasi-Trefftz DG method for diffusion--advection--reaction equations is implemented in NGSolve \cite{ngsolve} and the code is freely available.
	This paper is mainly based on the third author's master thesis \cite{perinati2023quasitrefftz}, where some more details can be found.
	
	\section{Polynomial quasi-Trefftz space}\label{s:QT} 
	In this section, we first introduce the polynomial quasi-Trefftz space for a general linear PDE with smooth  coefficients and right-hand side, and prove that it contains high-order approximations of all smooth PDE solutions.
	While the definition and the approximation properties of this space only require the governing PDE to have $C^{p-m}$-smooth coefficients and right-hand side, $p$ being the polynomial degree of the space and $m$ the order of the PDE, practical aspects of quasi-Trefftz also rely on a simple non-degeneracy assumption on the differential operator.
	In particular, the dimension of the quasi-Trefftz space depends on the differential operator, and we show that under assumption~\eqref{hpcoeff} this dimension is much reduced compared to standard polynomial spaces, as expressed in \eqref{eq:dimreduc}.
	In this case, we provide an algorithm for the construction of quasi-Trefftz functions, including for a non-zero-RHS PDE, and more specifically for the construction of a basis for a zero-RHS PDE.
	
	\subsection{Definitions and notation}\label{s:DefNot}
	Let $d \in \IN$ be the space dimension.
	Multi-indices are denoted $\bi:=(i_{1},\ldots,i_{d})
	\in \IN_0^{d}$, their length $|\bi|:=i_{1}+\cdots+i_{d}$, and $\leq$ denotes the partial order
	defined by $\bi\leq \bj$ if $i_k\leq j_k$ for all $k\in\{1,\dots,d\}$.
	As a reminder, the multi-index factorial and binomial coefficients are defined as
	$$\bi!:=i_{1}!\cdots i_{d}!, \qquad
	\binom\bi\bj:=\frac{\bi!}{\bj!(\bi-\bj)!}
	=\binom{i_{1}}{j_{1}}\cdots
	\binom{i_{d}}{j_{d}}.
	$$
	We use standard multi-index notation $D^\bi f := \partial_{x_1}^{i_{1}}\cdots\partial_{x_d}^{i_{d}}f$ for derivatives of a function $f$ of $\bx \in\IR^d$, and  $\bx^\bi=x_1^{i_{1}}\cdots x_d^{i_{d}}$ for monomials.
	
	Let $E\subset \IR^d$ be an open
	set.
	Denote by $\IP^p(E)$ the space of polynomials of degree at most $p\in\IN_0$ defined on $E$.
	The general linear partial differential operator of order $m\in\IN$, denoted $\mathcal M$, is expressed in terms of its variable coefficients $\alpha_{\bj}: E\to \IR$ for $\bj\in\IN_0^d$ and $\abs{\bj}\leq m$ as
	\begin{equation}\label{linearop}
	\calM:=\sum_{\bj\in\IN_0^d,\;|\bj|\leq m} \alpha_{\bj} D^{\bj}.
	\end{equation}
	The PDE of interest, for the unknown $u:E \to\IR$ and source term $f:E \to \IR$, then reads
	$\calM u=f$ in $E$.
	
	We introduce the polynomial quasi-Trefftz spaces for this PDE, under assumptions of smoothness of the right-hand side and the operator coefficients.
	\begin{Def}[Quasi-Trefftz space]\label{def:QT}
	Let $p\in\IN_0$, let $E \subset\IR^d$ be an open set and let $\bx^E\in E$.
	Assume that the coefficients $\alpha_\bj\in C^{\max\{p-m,0\}}(E)$ for all $\abs{\bj}\leq m$ and that $f\in C^{\max\{p-m,0\}}(E)$. 
	We define the \textit{polynomial quasi-Trefftz space} for the equation $\calM u=f$ in $E$ as
	\begin{equation}\label{inhomQT}
	\QT^p_f(E):=\big\{ v\in \IP^p(E) \mid D^{\bi} \calM v (\bx^E)=D^{\bi} f (\bx^E)\quad \forall \bi\in \IN^d_0,\ |\bi|\leq p-m\big\}.
	\end{equation}
	\end{Def}
	The choice of the maximal order $p-m$ for the derivatives of $\mathcal Mv-f$ that vanish at $\bx^E$ is optimal in the following sense: for a lower order the space would be larger but it would not have better approximation properties; for a higher order the space would not enjoy the same approximation properties; see \cite[Remark~4.4]{IGMS_MC_2021}.
	By definition, $\QT^p_f(E)$ is a subset of the full polynomial space  $\IP^p(E)$ and an affine space; when $f=0$, $\QT^p_0(E)$ is a vector space.
	For $p<m$, $\QT^p_f(E)$ coincides with $\IP^p(E)$, so we always assume $p\geq m$.
	
	\begin{rem}[Non nested spaces]
	In general,  $\QT^p_f(E)\not\subset\QT^{p+1}_f(E)$, i.e., for increasing polynomial degrees $p$, the quasi-Trefftz spaces are not nested.
	To see this, consider, for example, 
	the second-order diffusion--advection--reaction operator $\calM u:=-\Delta u+\bbeta\cdot\nabla u+\sigma u $ with $\bbeta(\bx)=(1,\dots,1)^\top$, $\sigma(\bx)=\frac{2}{x_1^2+1}$ and $f=0$.
	Choosing the point $\bx^E=\bzero$ and the function $v(\bx)=x_1^2+1\in  \IP^2(E)$,  then, $\calM v(\bx)=-2+2x_1+2$, so $\calM v(\bx^E)=0$.
	Hence $v\in \QT^2_0(E)$, but $\partial_{x_1}\calM v(\bx)=2$, implying that $v\in\QT^2_0(E)\setminus \QT^3_0(E)$.
	\end{rem}
	
	\begin{rem}[Constant-coefficients: Trefftz and quasi-Trefftz spaces]
	Let us consider a constant-coefficient differential operator $\calM$.
	When all the terms in \eqref{linearop} 
	are derivatives of the same order (i.e.\ $\alpha_\bj=0$ for $|\bj|<m$), such as, for example, in the Laplace and the wave equations, the polynomial Trefftz space 
	$\IT^p(E):=\{v\in \IP^p(E) \mid \calM v=0 \text{ in } E\}$
	approximates solutions of the homogeneous PDE $\calM u=0$ with the same orders of $h$-convergence as the full polynomial space $\IP^p(E)$ \cite[Lemma~1]{moiola2018space} (assuming $E$ is star-shaped).
	On the other hand, if the differential operator $\calM$ includes derivatives of different orders, the convergence rates for the polynomial Trefftz space can be lower.
	For example, for the linear time-dependent Schr\"{o}dinger equation, in \cite{gomez2023polynomial} the same
	rates are obtained for $\IP^p(E)$ and $\IT^{2p}(E)$, i.e.\ the Trefftz space requires doubling the polynomial degree.
	In the extreme case when a zero-order term is present, i.e.\ $\alpha_{\bzero}\ne0$, such as in the case of the Helmholtz equation, $\calM u=0$ does not admit polynomial solutions and the polynomial Trefftz space is trivial, $\IT^p(E)=\{0\}$.
	The quasi-Trefftz space, instead, is always rich enough to give the same approximation rates as the full $\IP^p(E)$, as we see below in \Cref{th:Approx}.
	This suggests that quasi-Trefftz methods could be an effective choice also for problems with piecewise-constant coefficients.
	\end{rem}
	
	\subsection{Approximation properties}\label{s:ApproxQT}
	For $q\in\IN_0$, the standard $C^q$ norms and seminorms are denoted by
	\begin{equation*}
	\N{v}_{C^0(E)}:=\sup_{\bx\in E}|v(\bx)|,\qquad
	\abs{v}_{C^q(E)}:=\max_{\bi\in\IN_0^{d},\; |\bi|=q}\N{D^\bi v}_{C^0(E)}.
	\end{equation*}
	Let $p\in\IN_0$ and let $\tayt^{p+1}_{\bx^E}[v]\in\IP^p(E)$ denote the \textit{Taylor polynomial} of order $p+1$ of $v\in C^p(E)$, centered at $\bx^E\in E$: 
	$$
	\tayt^{p+1}_{\bx^E}[v](\bx):=
	\sum_{|\bj|\le p}\frac1{\bj!}	D^\bj v(\bx^E)(\bx-\bx^E)^{\bj}.
	$$
	For every multi-index $\bi\in\IN_0^d$ with $|\bi|\le p$
	\begin{align}\nonumber
	D^\bi \tayt^{p+1}_{\bx^E}[v](\bx)=&
	\sum_{\substack{\abs{\bj}\leq p\\ \bj\ge \bi}}\frac{1}{\bj!}D^{\bj } v(\bx^E)\frac{\bj!}{(\bj-\bi)!}
	(\bx-\bx^E)^{\bj-\bi}
	=\sum_{\abs{\bk}\leq p-\abs{\bi}}\frac{1}{\bk!}D^{\bk+\bi } v(\bx^E)
	(\bx-\bx^E)^{\bk},
	\\\label{DiT=TDi}
	&\Longrightarrow \quad D^\bi \tayt^{p+1}_{\bx^E}[v](\bx)= \tayt^{p+1-\abs{\bi}}_{\bx^E}[D^{\bi}v](\bx).
	\end{align}
	From the evaluation of this identity at $\bx=\bx^E$ and $\tayt^{p+1}_{\bx^E}[v](\bx)\in\IP^p(E)$,
	it follows that
	\begin{equation}\label{DiT=Div}
	D^\bi \tayt^{p+1}_{\bx^E}[v](\bx^E)=
	\left\{
	\begin{array}{ll}
	D^\bi v(\bx^E) &\text{ if }|\bi|\le p,\\
	0 &\text{ if }|\bi|>p.
	\end{array}
	\right.
	\end{equation}
	Recall the Lagrange form of the Taylor remainder \cite[Cor.~3.19]{callahan2010advanced}: if $v\in C^{p+1}(E)$ and the segment $S$ with endpoints $\bx^E$ and $\bx$ is  contained in $E$, then exists $ \bx_*\in S$ such that
	\begin{equation}\label{eq:TaylorRem}
	v(\bx)-\tayt^{p+1}_{\bx^E}[v](\bx)=
	\sum_{|\bj|=p+1}\frac1{\bj!}D^\bj v(\bx_*)(\bx-\bx^E)^{\bj}.
	\end{equation}
	
	To prove the approximation properties of $\QT^p_f(E)$ and to construct quasi-Trefftz polynomials, we make the following regularity assumption on the PDE coefficients and right-hand side:
	\begin{equation}\label{eq:Cpm}
	p,m\in\IN, \qquad p\ge m, \qquad 
	\alpha_\bj\in C^{p-m}(E) \quad \text{for all }\bj\in\IN_0^d, \quad |\bj|\le m, \qquad 
	f\in C^{p-m}(E). 
	\end{equation}
	
	The following theorem provides the key approximation property of quasi-Trefftz spaces:  the orders of $h$-convergence for quasi-Trefftz spaces $\QT^p_f(E)$ are the same as those for full polynomial spaces $\IP^p(E)$ of the same degree.
	We denote the diameter of $E$ as $h_E:=\sup_{\bx,\by\in E}\abs{\bx-\by}$.

	\begin{theorem}\label{th:Approx}
	Under assumption~\eqref{eq:Cpm}, let
	$u\in C^{p+1}(E)$ satisfies $\calM u=f$ in $E$.
	Then, the Taylor polynomial $\mathsf{T}^{p+1}_{\bx^E}[u]\in \QT^p_f(E)$.
	
	Moreover, if $E$ is star-shaped with respect to $\bx^E$, then, for all $q\in\IN_0$ with 
	$q\le p$,
	\begin{equation}\label{eq:Approx}
	\inf_{v\in\QT^p_f(E)}\abs{u-v}_{C^q(E)}
	\ \le\ \abs{u-\tayt^{p+1}_{\bx^E}[u]}_{C^q(E)}
	\ \le\ \frac{d^{p+1-q}}{(p+1-q)!} h_E^{p+1-q} \abs{u}_{C^{p+1}(E)}.
	\end{equation}
	\end{theorem}
	\begin{proof}
	First we prove that $\tayt^{p+1}_{\bx^E}[u]\in \QT^p_f(E)$.
	By definition, $\tayt^{p+1}_{\bx^E}[u]\in \IP^p(E)$.
	Moreover, from the definition \eqref{linearop} of $\calM$ and the Leibniz product rule,
	for all $v\in C^p(E)$ and $ |\bi|\le p-m$ 
	\begin{equation}\label{derimage}
	D^\bi \calM v(\bx^E)=\sum_{\abs{\bj}\leq m} D^\bi \left(\alpha_{\bj} (\bx^E)D^\bj v(\bx^E)\right)
	=\sum_{\abs{\bj}\leq m}\sum_{\br\leq \bi} \binom{\bi}{\br} D^\br \alpha_{\bj} (\bx^E)D^{\bi-\br+\bj}v(\bx^E).
	\end{equation}
	Hence, with $v= \mathsf{T}^{p+1}_{\bx^E}[u]$, we have for all $|\bi|\le p-m$
	\begin{align*}
	\begin{split}
	D^\bi \calM \mathsf{T}^{p+1}_{\bx^E}[u](\bx^E)&=\sum_{\abs{\bj}\leq m}\sum_{\br\leq \bi} \binom{\bi}{\br} D^\br \alpha_{\bj}(\bx^E) D^{\bi-\br+\bj} \mathsf{T}^{p+1}_{\bx^E}[u](\bx^E)\\
	&=\sum_{\abs{\bj}\leq m}\sum_{\br\leq \bi} \binom{\bi}{\br} D^\br \alpha_{\bj}(\bx^E) D^{\bi-\br+\bj} u(\bx^E)
	=D^\bi\calM u (\bx^E) = D^\bi f (\bx^E) . 
	\end{split}
	\end{align*}
	The second equality
	follows from the property \eqref{DiT=Div} with partial derivatives of order at most equal to $\abs{\bi}+m\leq p$, while the third one is \eqref{derimage} again  with $v=u\in C^{p+1}(E)$.
	In the last step we use that $u$ is solution of $\calM u=f$ in $E$.
	This shows that the Taylor polynomial $\mathsf{T}^{p+1}_{\bx^E}[u]$ belongs to the quasi-Trefftz space $\QT^p_f(E)$.
	
	This immediately implies the first inequality in the best-approximation bound \eqref{eq:Approx}.
	To prove the second inequality, fix $q$ with $0\le q\le p$.
	Using the $\abs{ \cdot }_{C^q}$-seminorm definition, 
	the identity $D^\bi \tayt^{p+1}_{\bx^E}[u]=\tayt^{p+1-|\bi|}_{\bx^E}[D^\bi u]$ for $|\bi|=q\le p$ from \eqref{DiT=TDi}, 
	estimating the Lagrange form of the Taylor remainder \eqref{eq:TaylorRem}, which is applicable because $u\in C^{p+1}(E)$ and $E$ is star-shaped with respect to $\bx^E$,
	we obtain the assertion:
	\begin{align*}
	\abs{u-\tayt^{p+1}_{\bx^E}[u]}_{C^q(E)}
	&=\max_{\bi\in\IN_0^{d},\; |\bi|=q}\N{D^\bi (u-\mathsf{T}^{p+1}_{\bx^E}[u])}_{C^0(E)}
	\\
	&=\max_{\bi\in\IN_0^{d},\; |\bi|=q}\N{D^\bi u-\mathsf{T}^{p+1-q}_{\bx^E}[D^\bi u]}_{C^0(E)}
	\\
	&	\le \max_{\bi\in\IN_0^{d},\; |\bi|=q}
	\sum_{\substack{|\bj|=p+1-q}}
	\frac1{\bj!}\sup_{\bx,\bx_*\in E}\abs{	D^{\bi+\bj}u(\bx_*)(\bx-\bx^E)^\bj}
	\\&
	\le \frac{d^{p+1-q}}{(p+1-q)!} h_E^{p+1-q} \abs{u}_{C^{p+1}(E)}.
	\end{align*} 	
	We have used the formula 
	$\sum_{|\bj|=k}\frac1{\bj!}= \frac{d^k}{k!}$ with $k=p+1-q$, 
	obtained from the multinomial theorem 
	$(w_1+\dots+w_d)^k=\sum_{\bj\in\IN_0^d, \abs{\bj}=k} \frac{k!}{\bj!}\bw^{\bj}$
	by choosing $\bw=(1,\dots,1)$.
	\end{proof}
	
	Bound~\eqref{eq:Approx} is an $h$-approximation estimate: it ensures convergence of the approximation error to zero when the size of the domain $E$ decreases.
	For analytic functions whose seminorm sequence $p\mapsto|u|_{C^p(E)}$ increases at most exponentially, it ensures also $p$-convergence, namely convergence on a fixed $E$ when $p\to\infty$.
	
	\subsection{Construction of quasi-Trefftz functions}\label{s:ConstructionQT}
	Under assumption \eqref{eq:Cpm}, this section proposes an explicit procedure to construct quasi-Trefftz functions under a further non-degeneracy assumption on the differential operator $\calM$, namely that 
	\begin{equation}\label{hpcoeff}
	\alpha_{\bj^*}(\bx^E)\ne0 \ \text{ for } \bj^*=(m,0,\dots,0)=m\be_1.
	\end{equation}
	Here we denote by $\be_k\in\IR^d$ the elements of the canonical basis of~$\IR^d$, defined by $(\be_k)_l=\delta_{kl}$, $1\le k,l\le d$.
	Assuming instead that $\bj^*=m\be_k$ for any $k$ between $2$ and $d$ would allow for the same reasoning.
	Condition \eqref{hpcoeff} might be circumvented with a more algebraic approach to the construction of quasi-Trefftz functions, which is currently under development.
	
	Constructing a polynomial $v\in\QT^p_f(E)$ boils down to computing 
	the coefficients $\{a_{\bk},\bk\in \IN_0^d,\abs{\bk}\le p\}$ of its expansion as a linear combination of scaled monomials centered at $\bx^E\in E$:
	\begin{equation}\label{lincombmon}
	v(\bx)=\sum_{\bk\in\IN_0^{d},|\bk|\leq p} a_\bk \left(\frac{\bx-\bx^E}{h_E}\right)^{\bk},\ 
	\text{ from which }D^{\bk} v(\bx^E)=\frac{\bk! }{h_E^{\abs{\bk}}}a_{\bk}.
	\end{equation}
	In order to state
	the conditions $D^{\bi}\mathcal{M}v(\bx^E)=D^{\bi}f(\bx^E)$ for $\abs{\bi}\leq p-m$
	in terms of the coefficients $a_\bk$, we note from \eqref{derimage} that 
	\begin{equation*}
	\begin{aligned}
	D^\bi \calM v(\bx^E)
	&=\sum_{\abs{\bj}\leq m}\sum_{\br \leq \bi} \binom{\bi}{\br} D^{\br} \alpha_{\bj} (\bx^E) 
	\frac{({\bi-\br+\bj})! }{h_E^{\abs{{\bi-\br+\bj}}}}a_{{\bi-\br+\bj}}\\&
	=	\sum_{\abs{\bj}\leq m}\sum_{\bl \leq \bi} \binom{\bi}{\bi-\bl} D^{\bi-\bl} \alpha_{\bj} (\bx^E) 
	\frac{({\bl+\bj})! }{h_E^{\abs{{\bl+\bj}}}}a_{{\bl+\bj}}
	\qquad\qquad |\bi|\le p-m.
	\end{aligned}
	\end{equation*}
	Under assumption \eqref{hpcoeff}, each of these conditions for $\abs{\bi}\leq p-m$ can be equivalently stated as
	\begin{align}\label{recursiveformula}
	a_{\bi+m\be_1}=\frac{{h_E^{\abs{\bi}+m}}}{\alpha_{m\be_1}(\bx^E)(\bi+m\be_{1})!}
	\Bigg( D^{\bi}	f(\bx^E)-\!\!\!\sum_{{\substack{\abs{\bj}\leq m\\ \bl\leq \bi\\ (\bj,\bl)\neq(m\be_1,\bi) }}}\!\!\! \binom{\bi}{\bi-\bl}  D^{\bi-\bl} \alpha_{\bj} (\bx^E) 
	\frac{({\bl+\bj})! }{h_E^{\abs{{\bl+\bj}}}} a_{{\bl+\bj}}\Bigg),
	\end{align}
	dividing by $\alpha_{m\be_1}(\bx^E)\ne0$.
	Imposing \eqref{recursiveformula} following an order such that, at each step, all the $a_{\bl+\bj}$ appearing at the right-hand side are known, would provide an iterative formula to compute all the $a_{\bk}$ such that $k_1\geq m$. 
	
	Accordingly, we propose to start by fixing all the coefficients  $a_{\bk}=\frac{h_E^{\abs{\bk}}}{\bk!}D^\bk v(\bx^E)$ such that $k_1< m$. 
	This is  equivalent to choosing $m$ polynomials $\psi_r\in\IP^{p-r}(\IR)$ for $r=0,\dots,m-1$ such that $\partial_{x_1}^r v(x_1^E,\cdot)= \psi_r$, where $v(x^E_1,\cdot)$ denotes the restriction of $ v$ to the hyperplane $\{x_1=x^E_1\}$. 
	We call this set of functions $\{\psi_r,0\leq r<m\}$ the ``Cauchy data'' of $v$  in analogy to the case of the wave equation \cite{IGMS_MC_2021} (with $m=2$ and $x_1$ corresponding to the time variable). This step is referred to as the \textit{initialization}.
	
	Next, given the Cauchy data of $v$, we propose the following precise ordering of the multi-indices~$\bi$ via three nested loops to compute \textit{iteratively} the coefficients $a_{\bi+m\be_1}$ in \eqref{recursiveformula}: 
	\begin{itemize}
	\item we start by looping over the length $q=\abs{\bi}$ increasingly from $q=0$ to $q=p-m$;
	\item at fixed $q$, we loop over the first component $i_1$ of $\bi$, from $i_1=0$ to $i_1=q$;
	\item at fixed $q$ and $i_1$, we compute $a_{\bi+m\be_1}$ for the indices $(i_2,\dots,i_d)$ such that 
	$i_2+\cdots+i_d=q-i_1$, 
	in any arbitrary order.
	\end{itemize} 
	
	\Cref{Algo:general} summarizes the procedure comprised of the initialization and the iterative step. Does it fulfill the goal of constructing a quasi-Trefftz function? It does, as for fixed $q$ and $i_1$ all the  coefficients $a_{\bl+\bj}$ appearing in the right-hand side of \eqref{recursiveformula} are already known: 
	(1) for $\ell_1+j_1<m$ they are fixed from the initialization;
	(2) for $\ell_1+j_1\geq m$ and $|{\bl}+\bj|< |\bi|+m$ they are computed at a previous iteration of the outer loop for $q'<q$;
	(3)  for $\ell_1+j_1\geq m$, $|{\bl}+\bj|=|\bi|+m$  and $(\bj,\bl)\neq(m\be_1,\bi)$ they are computed at the same iteration of the outer loop, 
	but at a previous iteration of the second loop for $i'_1 =\ell_1+j_1<i_1+m $. 
	
	All the information about the PDE required by \Cref{Algo:general} is encoded in the values at $\bx^E$ of the partial derivatives of order up to $p-m$ of the coefficients $\alpha_\bj$, and of the right-hand side $f$.
	
	\begin{algorithm}[ht]
	{\sc Algorithm}\\
	\SetAlgoLined
	Data: $p(\geq m)$, $\bx^E$, $h_E$,  $D^{\bl}f(\bx^E)$ for  $\abs{\bl}\leq p-m$,  $D^{\bl}\alpha_{\bj}(\bx^E)$ for  $\abs{\bl}\leq p-m$ and $\abs{\bj}\leq m$.\\
	Fix coefficients $a_{r,k_2,\dots,k_d}$ by choosing polynomials $\psi_r\in \IP^{p-r}(\mathbb{R}^{d-1})$, for $r=0,\dots,m-1$.\\
	Construct $v\in\QT^p_f(E)$ as follows:\\
	\For{$q=0$ to $p-m$\qquad (loop across $\{\abs{\bi}=q\}$ hyperplanes $\nearrow$)}{
	\For{$ i _1=0$ to $q$ \qquad (loop across constant-$i_1$ hyperplanes $\rightarrow$)}{
	\For{ $(i_2,\dots,i_d)$ with $\abs{(i_2,\dots,i_d)}=q- i _1$\qquad
	}{
	\begin{align*}
	a_{\bi+m\be_1}=&\ \frac{{h_E^{q+m}}}{\alpha_{m\be_1}(\bx^E)(\bi+m\be_{1})!}\times \\&
	\Biggl( D^{\bi}	f(\bx^E)-\!\!\!\sum_{{\substack{\abs{\bj}\leq m\\ \bl\leq \bi\\ (\bj,\bl)\neq(m\be_1,\bi) }}}\!\!\! \binom{\bi}{\bi-\bl}  D^{\bi-\bl} \alpha_{\bj} (\bx^E) 
	\frac{({\bl+\bj})! }{h_E^{\abs{{\bl+\bj}}}} a_{{\bl+\bj}}\Biggl).
	\end{align*}
	}}}
	$\displaystyle v(\bx)=\sum_{\bk\in\IN_0^{d},|\bk|\leq p} a_\bk \bigg(\frac{\bx-\bx^E}{h_E}\bigg)^{\bk}$.
	\caption{The algorithm for the computation of the monomial expansion of any quasi-Trefftz polynomial $v\in\QT^p_f(E)$ given its Cauchy data $(\psi_r)_{r=0,\ldots,m-1}$.
	}
	\label{Algo:general}
	\end{algorithm}
	
	We will see in \cref{s:QTDGdiscretizazion} that, in order to treat non-homogeneous BVPs, we need to construct an elementwise approximate particular solution, i.e.\ an element of $\QT_f^p(E)$ for each mesh element $E$.
	To this purpose, it is sufficient to choose any Cauchy data $(\psi_r)_{r=0,\ldots,m-1}$ and apply \Cref{Algo:general}.
	In practice we will choose $\psi_r=0$ for all $r$.
	
 	Next we turn to the question of the uniqueness of the quasi-Trefftz polynomial with given Cauchy data. 
	
	\begin{prop}\label{prop:Uniqueness}
	Assume that the regularity and the non-degeneracy conditions \eqref{eq:Cpm} and \eqref{hpcoeff} are satisfied for an open, connected set $E\subset\IR^d$,
	and let $\bx^E\in E$. 
	Given any set of $m$ polynomials $\psi_r\in \IP^{p-r}(\mathbb{R}^{d-1})$ for $r=0,\dots,m-1$, there exists a unique $v\in\QT^p_f(E)$ such that $\partial_{x_1}^r v(x_1^E,\cdot)= \psi_r$ for $r=0,\dots,m-1$.
	\end{prop}
	\begin{proof}
	Any polynomial $v\in\QT^p_f(E)$ is uniquely determined by the sets of its coefficients  
	$\{a_\bk=\frac{h_E^{\abs{\bk}}}{\bk!}D^\bk v(\bx^E), \bk\in\mathbb N_0^d, |\bk|\le p\}$ as in \eqref{lincombmon}. 
	The corresponding index set can be split  as
	$$\left\{\bk\in\mathbb N_0^d,\; |\bk|\le p\right\}  
	= \left\{ \bk\in\mathbb N_0^d,\; |\bk|\le p,\; k_1<m\right\} \cup 
	\left\{\bk\in\mathbb N_0^d,\; |\bk|\le p,\; k_1\geq m\right\} .$$
	On the one hand, the first set of coefficients $a_\bk$ is uniquely defined by imposing $\partial_{x_1}^r v(x_1^E,\cdot)= \psi_r$ for $r=0,\dots,m-1$ since then $D^\bk v(\bx^E)=D^{(k_2,\dots,k_d)} \psi_{k_1}(x_2^E,\dots,x_d^E)$ for all $\bk$ with $k_1<m$.
	On the other hand, Algorithm \ref{Algo:general} shows that it is possible to compute the coefficients $a_\bk$ for all the indices $\bk$ in the second set, with $k_1\ge m$, thus there exists a $v\in\QT^p_f(E)$ as desired.
	This is unique because the coefficients of any quasi-Trefftz $v$ must satisfy equation~\eqref{recursiveformula}, thus each of those in the form $a_{\bi+m\be_1}$ are determined by the $a_\bk$ that appear earlier in the ordering given by the nested loops of \Cref{Algo:general}.
	In particular, if $v_1,v_2\in\QT_f^p(E)$ share the same first set of coefficients, then they coincide.
	\end{proof}
	The lowest-dimensional cases are ideal for a visual representation of the iterated loops in \Cref{Algo:general}.
	In the 1D case ($d=1$), only the outermost loop over $q=i$ is present, and the algorithm reduces to the sequential computation of $a_{i+m}$ from $i=0$ to $i=p-m$.
	In the 2D case ($d=2$), the innermost loop degenerates to the computation of the single coefficient $a_{i_1+m,q-i_1}$.
	\Cref{fig:IndexTriangle} illustrates the dependence between the coefficients $a_{\bk}$ of the monomial expansion of $v\in\QT^p_f(E)$ and their ordering as they are computed in \Cref{Algo:general} for $d=2$, $m=2$ and $p=6$.
	The dots in the quarter-plane of multi-indices $\bk=(k_1,k_2)\in\mathbb N_0^2$  represent the coefficients $a_{k_1,k_2}$.
	Under the constraint that $k_1+ k_2\le p$, these dots form a triangular shape in the plane.
	To initialize the algorithm we choose the Cauchy data, which consists of two functions $\psi_0\in \IP^p(\mathbb{R}^{d-1})$ and $\psi_1\in \IP^{p-1}(\mathbb{R}^{d-1})$ such that  $ v(x^E_1,\cdot)= \psi_0$ and  $	\partial_{x_1} v(x^E_1,\cdot)= \psi_1$.
	This choice determines the coefficients $a_{0, k _2}$ with  $ 0\leq k _2\le p$ and $a_{1, k _2}$ with  $0\leq  k _2\le p-1$, represented by the shaded yellow area in the figure.
	All the other coefficients are then uniquely determined and can be computed in the iterative part of the algorithm using relation \eqref{recursiveformula}.
	See \cite[Fig.~5.4]{perinati2023quasitrefftz} for a similar figure with $d=3$.
	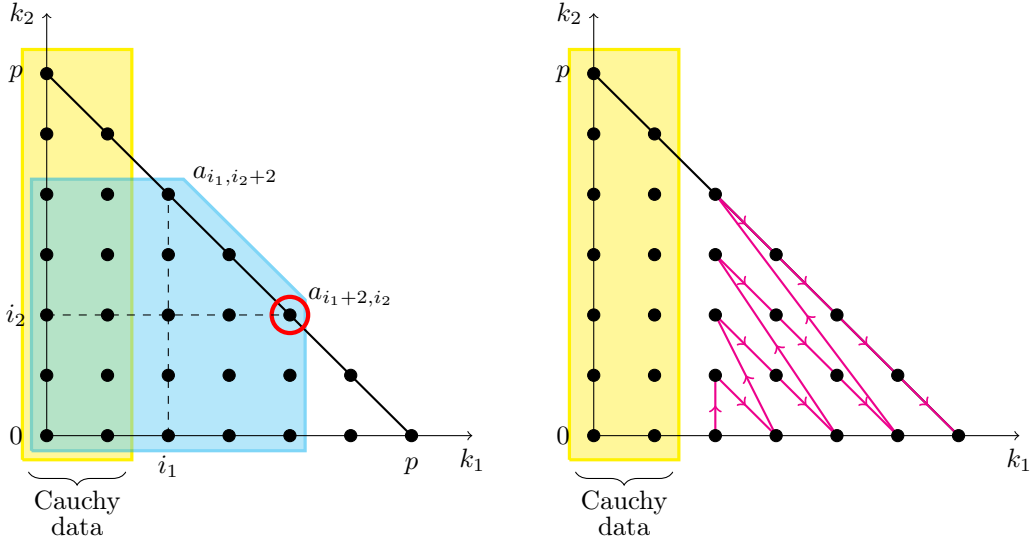
\begin{figure}[htb]\centering
	\newcommand{\YellInc}{.4}
	\newcommand{\BlueInc}{.25}
	\begin{tikzpicture}[scale=.8]
	\draw[yellow,very thick,fill=yellow!50!white](-\YellInc,-\YellInc)--(-\YellInc,6\YellInc)--(1\YellInc,6\YellInc)--(1\YellInc,-\YellInc)--(-\YellInc,-\YellInc);
	\draw[cyan,very thick,fill=cyan!50!white,opacity=0.5](-\BlueInc,-\BlueInc)--(-\BlueInc,4\BlueInc)--(2\BlueInc,4\BlueInc)--(4\BlueInc,2\BlueInc)--(4\BlueInc,-0\BlueInc)--(-\BlueInc,-\BlueInc);
	\draw[->](0,0)--(7,0); 
	\draw[->](0,0)--(0,7); 
	\draw[thick](0,6)--(6,0);
	\draw(-.5,0)node{$0$};
	\draw(-.5,6)node{$p$};
	\draw(6,-0.5)node{$p$};
	\draw[fill](0,0)circle(.1);
	\draw[fill](1,0)circle(.1);
	\draw[fill](2,0)circle(.1);
	\draw[fill](3,0)circle(.1);
	\draw[fill](4,0)circle(.1);
	\draw[fill](5,0)circle(.1);
	\draw[fill](6,0)circle(.1);
	\draw[fill](0,1)circle(.1);
	\draw[fill](1,1)circle(.1);
	\draw[fill](2,1)circle(.1);
	\draw[fill](3,1)circle(.1);
	\draw[fill](4,1)circle(.1);
	\draw[fill](5,1)circle(.1);
	\draw[fill](0,2)circle(.1);
	\draw[fill](1,2)circle(.1);
	\draw[fill](2,2)circle(.1);
	\draw[fill](3,2)circle(.1);
	\draw[fill](4,2)circle(.1);
	\draw[ultra thick,red](4,2)circle(.3);
	\draw(5,2.3)node{$a_{ i _1+2, i _2}$};
	\draw[fill](0,3)circle(.1);
	\draw[fill](1,3)circle(.1);
	\draw[fill](2,3)circle(.1);
	\draw[fill](3,3)circle(.1);
	\draw[fill](0,4)circle(.1);
	\draw[fill](1,4)circle(.1);
	\draw[fill](2,4)circle(.1);
	\draw(3.1,4.3)node{$a_{ i _1, i _2+2}$};
	\draw[fill](0,5)circle(.1);
	\draw[fill](1,5)circle(.1);
	\draw[fill](0,6)circle(.1);
	\draw(7,-.4)node{$ k _1$};
	\draw(-.4,7)node{$ k _2$};
	\draw[decorate, decoration={brace,amplitude=5}] (1.3,-.6) -- (-0.3,-.6);
	\draw(0.5,.-1.1)node{Cauchy};
	\draw(0.5,.-1.5)node{data};
	\draw[dashed](2,0)--(2,4); \draw(2,-.5)node{$ i _1$};
	\draw[dashed](0,2)--(4,2); \draw(-.5,2)node{$ i _2$};
	\end{tikzpicture}\qquad
	\begin{tikzpicture}[scale=0.8]
	\draw[yellow,very thick,fill=yellow!50!white](-\YellInc,-\YellInc)--(-\YellInc,6\YellInc)--(1\YellInc,6\YellInc)--(1\YellInc,-\YellInc)--(-\YellInc,-\YellInc);
	\draw[->](0,0)--(7,0); 
	\draw[->](0,0)--(0,7); 
	\draw[thick](0,6)--(6,0);
	\draw(-.5,6)node{$p$};
	\draw(-.5,0)node{$0$};
	\draw[middlearrow={>},thick,magenta] (2,0) to (2,1);
	\draw[middlearrow={>},thick,magenta] (2,1) to (3,0);
	\draw[middlearrow={>},thick,magenta] (3,0) to (2,2);
	\draw[middlearrow={>},thick,magenta] (2,2) to (3,1);
	\draw[middlearrow={>},thick,magenta] (3,1) to (4,0);
	\draw[middlearrow={>},thick,magenta] (4,0) to (2,3);
	\draw[middlearrow={>},thick,magenta] (2,3) to (3,2);
	\draw[middlearrow={>},thick,magenta] (3,2) to (4,1);
	\draw[middlearrow={>},thick,magenta] (4,1) to (5,0);
	\draw[middlearrow={>},thick,magenta] (5,0) to (2,4);
	\draw[middlearrow={>},thick,magenta] (2,4) to (3,3);
	\draw[middlearrow={>},thick,magenta] (3,3) to (4,2);
	\draw[middlearrow={>},thick,magenta] (4,2) to (5,1);
	\draw[middlearrow={>},thick,magenta] (5,1) to (6,0);
	\draw[fill](0,0)circle(.1);
	\draw[fill](1,0)circle(.1);
	\draw[fill](2,0)circle(.1);
	\draw[fill](3,0)circle(.1);
	\draw[fill](4,0)circle(.1);
	\draw[fill](5,0)circle(.1);
	\draw[fill](6,0)circle(.1);
	\draw[fill](0,1)circle(.1);
	\draw[fill](1,1)circle(.1);
	\draw[fill](2,1)circle(.1);
	\draw[fill](3,1)circle(.1);
	\draw[fill](4,1)circle(.1);
	\draw[fill](5,1)circle(.1);
	\draw[fill](0,2)circle(.1);
	\draw[fill](1,2)circle(.1);
	\draw[fill](2,2)circle(.1);
	\draw[fill](3,2)circle(.1);
	\draw[fill](4,2)circle(.1);
	\draw[fill](0,3)circle(.1);
	\draw[fill](1,3)circle(.1);
	\draw[fill](2,3)circle(.1);
	\draw[fill](3,3)circle(.1);
	\draw[fill](0,4)circle(.1);
	\draw[fill](1,4)circle(.1);
	\draw[fill](2,4)circle(.1);
	\draw[fill](0,5)circle(.1);
	\draw[fill](1,5)circle(.1);
	\draw[fill](0,6)circle(.1);
	\draw[decorate, decoration={brace,amplitude=5}] (1.3,-.6) -- (-0.3,-.6);
	\draw(0.5,.-1.1)node{Cauchy};
	\draw(0.5,.-1.5)node{data};
	\draw(7,-.4)node{$k_1$};
	\draw(-.4,7)node{$k_2$};
	\end{tikzpicture}
	\caption{Indices $\bk$ in the $(k_1,k_2)$-plane in the case $d=2$, $m=2$ and $p=6$.
	Each black dot $\bullet$ corresponds to the coefficient $a_{k _1, k _2}$ of the monomial expansion \eqref{lincombmon} of $v$.
	The indices $\bk$ with $ k _1\in\{0,1\}$ are highlighted in the shaded yellow area; the corresponding coefficients are determined by the Cauchy data $\psi_0,\psi_1$ of $v$.
	Left panel: the indices highlighted in the shaded blue area correspond to 
	the coefficients appearing in formula \eqref{recursiveformula} for computing $a_{\bi+2\be_1}$, with $\bi=(2,2)$, identified by the dot surrounded by the red circle \protect\tikz[red,ultra thick]\protect\draw(0,0) circle[radius=.14];.
	Right panel: illustration of the index ordering in Algorithm \ref{Algo:general}. 
	All coefficients with indices located in the non-shaded region are computed with formula \eqref{recursiveformula} 
	in a double loop: first across diagonals $\nearrow$, and then along each diagonal $\searrow$. The ordering is shown by the magenta arrows 
	\raisebox{1mm}{\protect\tikz[magenta,thick,->]\protect\draw(0,0)--(.4,0);}.}
	\label{fig:IndexTriangle}
	\end{figure}

\begin{rem}[Computational cost]
The main contribution to the computational cost of Algorithm~\ref{Algo:general} is due to the computation and the evaluation of the partial derivatives of the PDE coefficients $\alpha_\bj$, and of the source term $f$.
If these functions are analytic in the whole computational domain, it is possible to compute their derivatives symbolically only once, and then evaluate them elementwise, possibly in parallel.
In practice, the cost strongly depends on the format in which these data are available and on the implementation.
The total number of derivatives 
needed is bounded above by $\mathcal{O}((p-m+d)^{d} (m+d)^d)$.
The total number of operations required by the three loops of  Algorithm~\ref{Algo:general} is less than $\mathcal{O}((p-m+1)^{d+2} (p-m+d)^{d-2}(m+d)^d)$.
\end{rem}
	
	\subsection{Construction of a basis for the homogeneous equation}\label{s:BasisQT}
	In this section, we define, for any $p\in\IN $, a basis for the quasi-Trefftz space $\QT^p_0(E)$ for the homogeneous equation $\calM u=0$ and use \Cref{Algo:general} to explicitly construct it.
	We denote
	$$S_{d,p}:=\mathrm{dim}\big(\IP^p\big(\mathbb{R}^d)\big)=\binom{p+d}{d} \quad\text{ and }\quad 
	I_{d,p,m}:=\left\{(r,s)\in \IN_0^2\;
	\begin{tabular}{|l}
	$0\leq r\leq m-1,$\\
	$1\leq	s\leq S_{d-1,p-r}$
	\end{tabular}
	\right\}.$$
	In order to define a set of quasi-Trefftz functions, we first choose $m$ polynomial bases:
	\begin{equation*}
	\left\{\psi_{(r,s)}\right\}_{(r,s)\in I_{d,p,m}}
	\text{such that, }\forall r\in\{0,\ldots,m-1\}, \;
	\{\psi_{(r,s)}\}_{s=1,\ldots,S_{d-1,p-r}} \text{ is a}
	\text{ basis for }\IP^{p-r}(\mathbb R^{d-1}).
	\end{equation*}
	Their total cardinality is 
	\begin{equation}\label{eq:Ndp}
	N_{d,p}:=card(I_{d,p,m})=S_{d-1,p}+\dots+S_{d-1,p-m+1}=\binom{p+d-1}{d-1}+\dots+\binom{p+d-m}{d-1}.
	\end{equation}
	We then define the following set of $N_{d,p}$ elements of $\QT^p_0(E)$:
	\begin{equation}\label{eq:BEp}
	\mathcal{B}_E^p:=\left\{b_{(r,s)}\in\QT^p_0(E) \;
	\begin{tabular}{|l}
	$	\partial_{x_1}^{r}b_{(r,s)}(x^E_1,\cdot)=	\psi_{(r,s)}, $\\[0.2cm]
	$	\partial_{x_1}^{r'}b_{(r,s)}(x^E_1,\cdot)=0  \text{ for } r'=0,\dots, m-1, r'\neq r$
	\end{tabular}
	\right\}_{(r,s)\in I_{d,p,m}}.
	\end{equation}
	Equivalently, for each $(r,s)\in I_{d,p,m}$, the element $b_{(r,s)}$ is a polynomial of degree at most $p$ satisfying the quasi-Trefftz property and with prescribed Cauchy data:
	\begin{align*}
	\begin{cases}
	D^\bi  \calM b_{(r,s)}(\bx^E)=0 
	&\quad  \bi \in\IN_0^{d},\; |\bi|\le p-m,\\
	\partial_{x_1}^{r}b_{(r,s)}(x^E_1,\cdot)=	\psi_{(r,s)} \\
	\partial_{x_1}^{r'}b_{(r,s)}(x^E_1,\cdot)=0 &\quad r'=0,\dots, m-1,\; r'\neq r.
	\end{cases}
	\end{align*}
	Next we show that the set $\mathcal B_E^p$ of the elements $b_{(r,s)}$ for all $(r,s)\in I_{d,p,m}$ forms a basis of $\QT^p_0(E)$.
	\begin{prop}\label{prop:Basis}
	Assume that the regularity and non-degeneracy conditions \eqref{eq:Cpm} and \eqref{hpcoeff} are satisfied for an open, connected set $E\subset\IR^d$,
	and let $\bx^E\in E$. 
	Let $\{\psi_{(r,s)}\}_{s=1,\ldots,S_{d-1,p-r}}$ be a basis of $\IP^{p-r}(\mathbb R^{d-1})$ for each $r\in\{0,\ldots,m-1\}$.
	Then the set $\mathcal{B}^p_E$ in \eqref{eq:BEp} is a basis of the space $\QT^p_0(E)$.
	\end{prop}
	\begin{proof}
	Each $b_{(r,s)}\in\calB^p_E$ in \eqref{eq:BEp} is uniquely defined by
	\Cref{prop:Uniqueness}. 
	We need to verify that  $\mathcal{B}_E^p$ is a spanning set of linearly independent functions.
	
	For any $v\in\QT^p_0(E)$ and $r=0,\dots,m-1$, since the restriction to $\{x_1=x_1^E\}$ of the derivative $\partial^{r}_{x_1} v$ is a polynomial of degree $p-r$, there exist some coefficients $\{\lambda_{(r,s)}\}_{(r,s)\in I_{d,p,m}}\subset \mathbb{R}$ such that
	\begin{align*}
	\partial^{r}_{x_1} v(x^E_1,\cdot)
	=\sum_{s=1}^{S_{d-1,p-r}}\lambda_{(r,s)}\psi_{(r,s)}
	=\sum_{s=1}^{S_{d-1,p-r}}\lambda_{(r,s)}\partial^{r}_{x_1}b_{(r,s)}(x^E_1,\cdot)
	=\partial^{r}_{x_1}
	\bigg(\underbrace{\sum_{s=1}^{S_{d-1,p-r}}\lambda_{(r,s)}b_{(r,s)}}_{=:w_r}\bigg)
	(x^E_1,\cdot).
	\end{align*} 
	Set $w:=\sum_{r=0}^{m-1}w_r=\sum_{(r,s)\in I_{d,p,m}}\lambda_{(r,s)}b_{(r,s)}$.
	By \eqref{eq:BEp}, $\partial^{r'}_{x_1} w_r(x_1^E,\cdot)=0$ for all $r'\ne r$, 
	thus $\partial^r_{x_1} w(x_1^E,\cdot)=\partial^r_{x_1} w_r(x_1^E,\cdot)=\partial^r_{x_1} v(x_1^E,\cdot)$ for all $r=0,\ldots,m-1$.
	Hence, $v$ and $w$ are both elements of $\QT^p_0(E)$ and they coincide by \Cref{prop:Uniqueness}, so that $v$
	is indeed a linear combination of $b_{(r,s)}$. 
	This proves that $\mathcal{B}_E^p$ is a spanning set for $\QT^p_0(E)$.
	
	Next we show that the  polynomials $\{b_{(r,s)} \}_{(r,s)\in I_{d,p,m}}$ are linearly independent.
	Assume that $\sum_{(r,s)\in I_{d,p,m}}c_{(r,s)} b_{(r,s)}=0$ for some coefficients $\{c_{(r,s)}\}_{(r,s)\in I_{d,p,m}}\subset \mathbb{R}$.
	Then, fixing any $\tilde{r}\in\{0,\dots,m-1\}$ and restricting to $\{x_1=x^E_1\}$, we obtain
	\begin{equation*}
	0=\sum_{(r,s)\in I_{d,p,m}}c_{(r,s)} \partial_{x_1}^{\tilde{r}} b_{(r,s)}(x^E_1,\cdot)
	=\sum_{s=1}^{S_{d-1,p-\tilde r}}
	c_{(\tilde{r},s)} \partial_{x_1}^{\tilde{r}} b_{(\tilde{r},s)}(x^E_1,\cdot)
	=\sum_{s=1}^{S_{d-1,p-\tilde r}}c_{(\tilde{r},s)}\psi_{(\tilde{r},s)}.
	\end{equation*}
	This implies that $c_{(\tilde{r},s)}=0$ for each $(\tilde{r},s)\in I_{d,p,m}$, since 
	$\{\psi_{(\tilde{r},s)}\}_{s=1,\ldots, S_{d-1,p-\tilde r}}$ are linearly independent. It concludes the proof.
	\end{proof}
	
	\Cref{prop:Basis} implies
	that the conditions in the definition of $\QT^p_0(E)$ are linearly independent:
	\begin{equation*}
	\dim\big(\IP^p(E)\big)-card\{\bi \in\IN_0^{d}\mid |\bi |\leq p-m\}=
	\binom{p+d}d-\binom{p+d-m}d
	=N_{d,p}=\dim\big(\QT^p_0(E)\big).
	\end{equation*}
	The equality between $\binom{p+d}d-\binom{p+d-m}d$ and the sum in \eqref{eq:Ndp} follows from manipulations of the binomials and the formula $ \sum_{k=0}^{n}\binom{r+k}{k}=\binom{r+n+1}{n}$ for $n,r\in \IN_0 $, under the assumption that $p\ge m$.
	In particular, we have 
	\begin{equation*}
	\dim\big(\QT^p_0(E)\big)=N_{d,p}
	=\begin{cases}
	m & d=1,\\ 
	m\left(p-\frac{m}{2}+\frac{3}{2}\right) & d=2,\\ 
	m\left(\frac12p^2+2p+\frac{11}{6}-\frac12mp-m+\frac{m^2}{6}\right) & d=3.
	\end{cases}
	\end{equation*} 
	For $m=2$, this expression simplifies to $N_{2,p}=2p+1$ and $N_{3,p}=(p+1)^2$.
	This means that, for second-order PDEs, $\QT^p_0(E)$ has the same dimension of the space of harmonic polynomials in $\IR^d$ of degree at most $p$, see \Cref{tab:DimQT}.
	In the one-dimensional case, when increasing the polynomial degree $p$ the dimension of the quasi-Trefftz space remains the same, but the space changes; see \cite[Fig.~5.1]{perinati2023quasitrefftz} for an example.

	\begin{table}[htb]
	\centering
	\begin{tblr}{colsep = 0.9mm,colspec={|c|ccc|ccc|ccc|ccc|ccc|ccc|ccc|}}
	\hline
	$p$ & \SetCell[c=3]{c}$2$ &&&  \SetCell[c=3]{c} $ 3$ &&&  \SetCell[c=3]{c}$4$ &&&   \SetCell[c=3]{c}$5$ &&& \SetCell[c=3]{c} $6$ &&&   \SetCell[c=3]{c}$10$ &&& \SetCell[c=3]{c} $20$ \\
	\hline
	$d=1$ & 2& 3&  \SetCell{bg=olive9}1.5 &
	2& 4& \SetCell{bg=olive9}2&
	2& 5& \SetCell{bg=olive9}2.5&
	2& 6& \SetCell{bg=olive9}3&
	2&7& \SetCell{bg=olive9}3.5&
	2& 11 & \SetCell{bg=olive9}5.5&
	2& 21 & \SetCell{bg=olive9}10.5\\
	$d=2$&
	5&6&\SetCell{bg=olive9}1.2&
	7&10&\SetCell{bg=olive9}1.43&
	9&15&\SetCell{bg=olive9}1.67&
	11&21&\SetCell{bg=olive9}1.91&
	13&28&\SetCell{bg=olive9}2.15&
	21&66&\SetCell{bg=olive9}3.14&
	41&231&\SetCell{bg=olive9}5.63\\
	$d=3$&
	9&10&\SetCell{bg=olive9}1.11&
	16&20&\SetCell{bg=olive9}1.25&
	25&35&\SetCell{bg=olive9}1.4&
	36&56&\SetCell{bg=olive9}1.56&
	49&84&\SetCell{bg=olive9}1.71&
	121&286&\SetCell{bg=olive9}2.36&
	441&1771&\SetCell{bg=olive9}4.02\\ \hline
	\end{tblr}
	\caption{The dimensions $\dim(\QT_0^p(E))$, $\dim(\IP^p(E))$, and the ratio $\frac{\dim(\IP^p(E))}{\dim(\QT_0^p(E))}$ for $m=2$.}
	\label{tab:DimQT}
	\end{table}

	Comparing against the dimension of the full polynomial space $\IP^p(E)$, we observe that
	\begin{align}\label{eq:dimreduc}
	\dim\big(\QT^p_0(E)\big)
	=\calO_{p\to\infty}(p^{d-1})\quad\ll\quad\dim\big(\IP^p(E)\big)
	=\binom{p+d}{d}=\calO_{p\to\infty}(p^{d}).
	\end{align}
	Thus, for large polynomial degrees $p$, the dimension of the quasi-Trefftz space is much smaller than the dimension of the full polynomial space of the same degree.
	
	Combined with \Cref{th:Approx}, this implies that smooth solutions of PDEs with smooth coefficients are approximated by $\QT^p_0(E)$ and by $\IP^p(E)$ with the same convergence rates with respect to the meshsize $h$, but with significantly less degrees of freedom in the quasi-Trefftz case.
	
	For $f\ne0$, the space $\QT_f^p(E)$ is not a linear space but an affine one.
	Given any $v_f\in \QT_f^p(E)$, which can be constructed using \Cref{Algo:general} with any choice of Cauchy data, we have $\QT_f^p(E)=v_f+\QT^p_0(E)$, therefore $\dim(\QT_f^p(E))=\dim(\QT^p_0(E))=N_{d,p}$.
	
	\section{Diffusion--advection--reaction equation}\label{s:ProblemDAR}
	Let $\Omega$ be an open, bounded, Lipschitz subset of $\mathbb{R}^d$ and denote by $\Gamma:=\partial\Omega$ its boundary. 
	We define the second-order, linear diffusion--advection--reaction operator $\calL$, applied to  $v:\Omega\to \mathbb{R}$, as
	\begin{equation}\label{daroperator}
	\mathcal{L}v:=	\mathrm{div}\left(-\bK  \nabla v  +\bbeta  v \right) +\sigma v,
	\end{equation}
	with coefficients
	$\bK:\Omega\to \mathbb{R}^{d\times d}$, $\bbeta:\Omega\to \mathbb{R}^d$ and $\sigma:\Omega\to \mathbb{R}$.
	
	Let $\Gamma_{\mathrm D}$ and $\Gamma_{\mathrm N}$ be sufficiently regular subsets of the boundary such that $\Gamma_{\mathrm D}\neq\emptyset$, $\Gamma=\Gamma_{\mathrm D} \cup \Gamma_{\mathrm N}$ and  $\Gamma_{\mathrm D}\cap\Gamma_{\mathrm N}=\emptyset$. Dirichlet and Neumann boundary conditions are imposed on $\Gamma_{\mathrm D}$ and $\Gamma_{\mathrm N}$, respectively.
	Let $\bn(\bx)$ be the outward unit normal vector to the boundary at $\bx\in\Gamma$.
	
	Let $f\in L^2(\Omega)$, $g_{\mathrm D}\in H^{\frac12}(\Gamma_{\mathrm D})$ and $g_{\mathrm N}\in L^2(\Gamma_{\mathrm N})$.
	We consider the following boundary value problem for the diffusion--advection--reaction equation:
	\begin{subequations}\label{eq:BVP}
	\begin{alignat}{2}
	\mathrm{div}(-\bK \nabla u +\bbeta u) +\sigma u &= f  && \quad \text{in  }\Omega,\label{eq}\\
	u&=g_{\mathrm D} && \quad\text{on }\Gamma_{\mathrm D}, \label{Dirichlet}\\
	- \bK \nabla u  \cdot \bn &=g_{\mathrm N} && \quad \text{on }\Gamma_{\mathrm N}.\label{Neumann}
	\end{alignat}
	\end{subequations}
	We make the following assumptions on the data:
	\begin{equation}\label{assumptionbeta}
	\bK=\bK^\top\in\left[L^{\infty}(\Omega)\right]^{d\times d},\quad \bbeta\in \left[W^{1,\infty}(\Omega)\right]^{d},\quad \sigma \in L^{\infty}(\Omega).
	\end{equation}
	In particular, this implies 
	$\bbeta\in H(\mathrm{div};\Omega).$
	We will write $\N{\bK}^2_{L^{\infty}(\Omega)}$ for the $L^\infty(\Omega)$ norm of the 2-norm of the matrix $\bK$, i.e.\ its spectral radius.
	We also assume that the \textit{ellipticity condition} is satisfied, i.e.\ there exists a constant $\kmin >0$ such that
	\begin{equation}\label{ellipticity}
	\bxi^\top\bK(\bx) \bxi \ge \kmin  \N{\bxi}^2\qquad \forall \bxi\in \mathbb{R}^d, 
	\;\text{a.e.\;}\bx\in \Omega,
	\end{equation}
	where $ \N{\cdot}$ denotes the Euclidean norm in $\mathbb{R}^d$.
	Choosing $\bxi=(1,0,\dots,0)^\top$ in \eqref{ellipticity} implies
	\begin{equation}\label{assalg}
	\bK_{11}(\bx)\ge \kmin >0\quad  \text{ a.e. }\bx\in\Omega.
	\end{equation}
	Under the ellipticity condition, $\calL$ is a non-degenerate second-order partial differential operator; in particular, \eqref{assalg} implies \eqref{hpcoeff} with $m=2$ and $\bj^*=2\be_1$ if the PDE coefficients are sufficiently smooth. 
	Moreover, we make the following assumption: if at least one among $\bbeta$ and $\sigma$ is not null, then there exists a constant $\sigma_0>0$ such that 
	\begin{equation}\label{assumptiondiv}
	\sigma(\bx)+\frac12\mathrm{div}\big(\bbeta(\bx)\big)\ge \sigma_{0}  \quad \text{ a.e. } \bx \in\Omega.
	\end{equation}
	
	When the advection term $\bbeta$ is non-zero, we distinguish between the inflow and outflow parts of the boundary $\Gamma$, defined as 
	\begin{equation}\label{inflowoutflow}
	\Gamma_{-}:=\{\bx\in\Gamma \mid \bbeta(\bx) \cdot \bn(\bx)<0 \}, \qquad
	\Gamma_{+}:=\{\bx\in\Gamma \mid \bbeta(\bx) \cdot \bn(\bx)\geq0\},
	\end{equation}
	respectively.
	Following e.g.~\cite[Thm.~3.8(iii)]{ern2004theory} and \cite[p.~2135]{houston2002discontinuous}, we assume that 
	$\bbeta\cdot \bn\ge 0$ on $\Gamma_N$ when $\Gamma_N$ is nonempty:
	\begin{equation}\label{gammaDgammaN}
	\Gamma_N\subset\Gamma_{+}, \qquad\text{equivalently,}\qquad \Gamma_- \subset\Gamma_D.
	\end{equation}
	This is done for simplicity but is also physically reasonable, for example, to model the movement of a substance knowing its concentration at the flow entrance but not at the exit.
	
	The classical variational formulation of problem \eqref{eq:BVP} is described in \cite[Chap.~3]{ern2004theory}. 
	In particular, \cite[Thm.~3.8]{ern2004theory} proves that, under these assumptions, \eqref{eq:BVP} admits a unique weak solution $u\in H^1(\Omega)$.

	\section{Discontinuous Galerkin discretization}\label{s:DG}
	\subsection{Mesh assumptions and notation}\label{s:Mesh}
	We assume that the domain $\Omega$ is a polytope of $\mathbb{R}^d$. 	
	We define polytopes by induction: a $0$-dimensional polytope is a subset of $\mathbb{R}^d$ containing a single point.
	For $n\in \IN$, $1\leq n\leq d$, a \textit{$n$-dimensional polytope} of $\mathbb{R}^d$ is a relatively open, bounded, connected and Lipschitz subset of a $n$-dimensional affine subspace of $\mathbb{R}^d$, such that its relative boundary is a finite union of \textit{$(n-1)$-facets}, i.e., closures of $(n-1)$-dimensional polytopes.
	For $n=1,2,3$, polytopes are simply segments, polygons and polyhedra, respectively.
	
	We discretize the domain $\Omega$ using a \textit{polytopal mesh} $\calT_h$, where each \textit{mesh element} $E\in\calT_h$ is a $d$-dimensional polytope  with  \textit{diameter} $h_E := \sup_{\bx,\by\in E} \abs{\bx-\by}$ and the \textit{meshsize} is $h:= \sup_{E\in\calT_h}h_E$.
	To analyze the DG method $h$-convergence, we consider a mesh sequence $\calT_{\mathcal{H}}:=\{\calT_h\}_{h\in\mathcal{H}}$ where $\mathcal{H}$ is a countable subset of $\{h\in\IR\mid h>0\}$ having only $0$ as accumulation point.
	
	For $ E\in\calT_h$, we denote by $\rho_E$ the \textit{radius} of the largest ball inscribed in $E$,
	and by $\abs{E}$ its $d$-dimensional \textit{measure}.
	The boundary of $E$ is indicated by $\partial E $ and  its $(d-1)$-dimensional measure by $\abs{\partial E}$. We define $\bn_E$ on  $\partial E$ as the \textit{unit outward normal} vector to the element $E$.
	
	We consider \textit{conforming meshes}: for all $E, E'\in\calT_h$, $E\neq E'$, the intersection $\partial E\cap \partial E'$ is either empty or a common $n$-dimensional facet with $n\leq d-1$.
	Distinct facets of $E$ may be co-planar.
	
	A \textit{mesh facet} is a $(d-1)$-facet of a polytopal mesh element $E\in\calT_h$, i.e.\ the closure of a $(d-1)$-dimensional polytope that is part of the boundary $\partial E$.
	We denote by $\calF_h$ the set of all facets of $\calT_h$.
	We assume that each $F\in\calF_h $ is either an \textit{interior facet} for which there exist two distinct elements $E_1, E_2\in \calT_h$ such that $F=\partial E_1\cap \partial E_2$, or a \textit{boundary facet} for which there exists an element $E\in\calT_h$ such that $F\subset \partial E\cap \partial \Omega$.
	The sets of interior and boundary facets are denoted by $\calF_h^{\mathrm I}$ and $\calF_h^{\mathrm B}$, respectively.
	We assume that it is possible to collect the boundary facets where Dirichlet conditions are assigned in a set, denoted $\calF_h^{\mathrm D}$, and the boundary facets where Neumann conditions are assigned in another set, denoted  $\calF_h^{\mathrm N}$.
	Similarly, $\calF_h^-$ and  $\calF_h^+$ denote the sets of inflow and outflow boundary facets.
	Thus $\calF_h=\calF_h^{\mathrm I}\cup\calF_h^{\mathrm D}\cup\calF_h^{\mathrm N}
	=\calF_h^{\mathrm I}\cup\calF_h^-\cup\calF_h^+$, where all unions are disjoint.
	For $F\in\calF_h$, we denote by	$h_F$ the \textit{diameter} of the facet $F$, by $\abs{F}$ its $(d-1)$-dimensional \textit{measure} and we associate to it a \textit{unit normal} vector $\bn_F$. 
	If $F\in \calF_h^{\mathrm B}$ then $\bn_F$ is chosen equal to $\bn$, i.e.\ pointing outward from $\Omega$.
	For each element $E\in\calT_h$ we define the set of all its facets as 
	$\calF_E:=\{F\in\calF_h\mid F\subset\partial E\}$.
	The maximum number of mesh facets composing the boundary of a mesh element is denoted by
	\begin{equation}\label{Npartial}
	N_{\partial}:=\max_{E\in\calT_h}card(\calF_E).
	\end{equation}

	We assume to work with mesh sequences that satisfy the following properties:
	\begin{enumerate}[label=(\roman*)]
	
	\item\label{it:starshaped} \textit{Star-shaped property}:
	there exists $0< r_\star\leq\frac12$ such that,  for all $h\in\mathcal{H}$, each $ E\in\calT_h$ is star-shaped with respect to a ball centered at some $\bx\in E$ and with radius $r_\star h_E$.
	\item\label{it:Graded}
	\textit{Graded mesh}(\cite[p.~744]{arnold1982interior}):
	there exists $\Cg >0$ such that,  for all $h\in\mathcal{H}$, for all $ E\in\calT_h $ and for all $F\in\calF_E$,
	\begin{equation}\label{gradedmesh}
	h_E\le\Cg h_F.
	\end{equation}
	\end{enumerate}
	The star-shaped property \ref{it:starshaped} implies the classical shape-regularity property (e.g.\ \cite[Def.~1.38(i)]{di2011mathematical}):
	\begin{equation}\label{shaperegularity}
	h_E\leq \Csr\rho_E, \qquad \text{with }\Csr =r_\star^{-1}.
	\end{equation}
	
	The star-shaped property \ref{it:starshaped} is used in the DG stability analysis of \cref{s:WellP}, while the graded-mesh condition \ref{it:Graded} is only used to prove quasi-Trefftz convergence rates in \Cref{th:finalerror}.
	The star-shaped property \ref{it:starshaped} implies also 
	the ``chunkiness'' of the mesh sequence, which will be used
	in the proof of Theorem \ref{th:finalerror}.
	\begin{lemma}[Chunkiness]\label{lemma:ChunkyStar}
	Let $E\subset\IR^d$ be a polytope with diameter $h_E$ that is star-shaped with respect to an open ball $B$ of radius $\rho_\star h_E$, for $0<\rho_\star\leq\frac12$.
	Then, 
	\begin{equation}\label{eq:chunky}
	h_E|\partial E|\le \frac d{\rho_\star}|E|.
	\end{equation}
	\end{lemma}
	\begin{proof}
	Assume without loss of generality that $B$ is centered at the origin $\bzero$.
	For each $(d-1)$-dimensional facet $F\in \calF_E$ of $E$, define $Y_F:=\{\by=t\bx\mid \bx\in F,\;0\le t<1\}$, the $d$-dimensional pyramid with basis $F$ and apex at the origin.
	By the star-shapedness with respect to the origin of $E$, we have that $E=\bigcup_{F\in\calF_E}Y_F$ and that $Y_{F_1}\cap Y_{F_2}$ has zero $d$-dimensional measure for different facets $F_1,F_2\in\calF_E$.
	The $d$-dimensional measure of $Y_F$ is $|Y_F|=\frac1d H_F|F|$, where the pyramid height $H_F$ is the distance between the hyperplane $\Pi_F$ containing $F$ and the origin (special cases are the usual triangle area formula ``half base times height'', and the 3D pyramid volume ``one third base area times height'').
	Since $\Pi_F$ contains a boundary facet and $E$ is star-shaped with respect to $B$, $\Pi_F$ cannot intersect $B$, thus $H_F\ge\rho_\star h_E$.
	Then the assertion follows:
	$$
	\frac{h_E|\partial E|}{|E|}=\frac{h_E\sum_{F\in\calF_E}|F|}{\sum_{F\in\calF_E}|Y_F|}
	=\frac{dh_E\sum_{F\in\calF_E}|F|}{\sum_{F\in\calF_E}H_F|F|}
	\le\frac{dh_E}{\inf_{F\in\calF_E}H_F}
	\le\frac d\rho_\star.
	$$
	\end{proof}
	Inequality \eqref{eq:chunky} is an equality when each facet of $E$ belongs to a hyperplane tangential to the ball $B$; this is the case, e.g., for all simplices, hypercubes, regular polygons and regular polyhedra.
	
	To apply the quasi-Trefftz approximation result of \Cref{th:Approx}, $E$ has to be star-shaped with respect to the point $\bx^E$ used to define the local discrete space $\QT^p_f(E)$: this point need not be the center of the ball in \Cref{lemma:ChunkyStar}.
	
	\Cref{lemma:ChunkyStar} ensures that, under assumption \ref{it:starshaped}, inequality \eqref{eq:chunky} holds for all $E\in\calT_h$ with $\rho_\star=r_\star$.

	We recall the definition of the \textit{broken Sobolev spaces}:
	\begin{align*}
	H^{m}(\calT_h):=&\{\varphi\in L^{2}(\Omega) \mid \varphi_{|_E} \in H^{m}(E) \quad \forall E \in \calT_h \}, \quad m\in\IN_0,\\
	H(\mathrm{div};\calT_h):=&\{\bw  \in[ L^{2}(\Omega)] ^{d} \mid 
	\bw_{|_E} \in H(\mathrm{div};E)  \quad \forall E \in \calT_h \}.
	\end{align*}
	
	We use the standard DG notation \cite[(2.5)--(2.7)]{ayuso2009discontinuous} for averages $\mvl{\cdot}$ and jumps $\jmp{\cdot}$ of any scalar function $\varphi\in H^1(\calT_h)$ and any vector-valued function  $\bw\in [H^1(\calT_h)]^d$ across the mesh facets: 
	\begin{align*}
	&\begin{cases}
	\begin{aligned}
	\mvl{\varphi}:=&\ \frac{\varphi_{|_{E_1}}+\varphi_{|_{E_2}}}2, &\qquad	
	\mvl{\bw}:=&\ \frac{\bw_{|_{E_1}}+\bw_{|_{E_2}}}2,\\
	\jmp{\varphi}:=&\ \varphi_{|_{E_1}}\bn_{E_1}+\varphi_{|_{E_2}} \bn_{E_2},&\qquad
	\jmp{\bw}:=&\ \bw_{|_{E_1}}\cdot\bn_{E_1}+\bw_{|_{E_2}} \cdot\bn_{E_2},
	\end{aligned}  \quad  \text{on }  F=\partial E_1 \cap \partial E_2,
	\end{cases}
	\\&\begin{cases}
	\begin{aligned}
	\mvl{\varphi} :=&\ \varphi_{|_E}, &\qquad	\mvl{\bw} :=&\ \bw_{|_E},\\
	\jmp{\varphi}:=&\ \varphi_{|_E}\bn_E,& \qquad \jmp{\bw}:=&\ \bw_{|_E}\cdot \bn_E,
	\end{aligned} \qquad  \text{on }  F\subset\partial E \cap \partial \Omega.
	\end{cases}
	\end{align*}
	We will use the ``DG magic formula'' \cite[Prop.~2.2.5]{perinati2023quasitrefftz}:
	for all $\varphi\in H^1(\calT_h)$ and for all $\bw\in[H^1(\calT_h)]^d$,
	\begin{equation}\label{DG_magic}
	\sum_{E\in \calT_h} \int_{\partial  E} \bw \cdot \bn_{ E} \varphi =
	\sum_{ F\in \calF_h^{\mathrm I}}\int_{ F} \big(\mvl{\bw} \cdot \jmp{\varphi} +  \jmp{\bw}  \mvl{\varphi}\big)
	+\int_{\partial\Omega}  \bw \cdot \bn \varphi.
	\end{equation}
	
For $p\in\IN_0$, we define the \textit{broken polynomial space} of degree at most $p$ on the mesh as
$\IP^p(\calT_h):=\{v\in L^2(\Omega) \mid  v_{|_E} \in \IP^p(E)  \quad \forall E \in \calT_h \}$.
For mesh sequences $\calT_{\mathcal{H}}$ enjoying the star-shaped property~\ref{it:starshaped}, the following \textit{discrete inverse trace inequality} holds: given $p\in\IN_0$,
\begin{equation} \label{discretetraceinequality}
\N{v}^2_{L^2(\partial E)}\leq  \frac{(p+1)(p+d)}{ r_\star} h_E^{-1} \N{v}^2_{L^2(E)}
\qquad \forall h\in\mathcal{H}, \quad E\in\calT_h, \quad v\in \IP^p(E).
\end{equation}
Indeed, under the star-shapedness assumption \ref{it:starshaped}, each mesh element can be partitioned as $E=\bigcup_{F\in\calF_E}Y_F$, where $Y_F$ is the  $d$-dimensional pyramid with basis $F$ and apex at the center of the ball in \ref{it:starshaped}. 
Then, inequality \eqref{discretetraceinequality} follows applying \cite[eq.~(3.4)]{CangianiDGH2017} (first proved in \cite{warburton2003constants} for arbitrary $d$) to each $Y_F$, using that the height of $Y_F$ is at least $r_\star h$, and that $|Y_F|\ge\frac1d r_\star h_E|F|$ as in the proof of Lemma~\ref{lemma:ChunkyStar}.
 Since in the DG scheme we will use the quasi-Trefftz space, which is a subset of the full polynomial space, this inequality can be applied.
	
	\subsection{Discontinuous Galerkin formulation}\label{s:DGformulation}
	We describe the DG variational formulation of the dif\-fu\-sion--ad\-vec\-tion--reaction problem \eqref{eq:BVP}.
	We consider the Symmetric Interior Penalty Galerkin (SIPG) method \cite{arnold1982interior} to handle the diffusion term, and the upwind DG method to handle the advection--reaction terms, following mostly 
	\cite[sect.~2.3]{di2011mathematical}.
	
	We define the DG scheme and carry out the abstract error analysis for a general discrete sub\-space $V_h$ of the broken polynomial space $\IP^p(\calT_h)$. 
	We will choose a global quasi-Trefftz space in \cref{s:QTDGdiscretizazion} and prove convergence rates for it.
	Following the non-conforming analysis of \cite[Thm.~1.35]{di2011mathematical} we define 
	$$
	V_*:= H^1(\Omega) \cap H^2(\calT_h), \qquad V_{*h} := V_* + V_h.
	$$
	
	Let $u$ be the weak solution of problem \eqref{eq:BVP}.
	We assume $u\in V_*$, which is guaranteed e.g.\ if $\Gamma_{\mathrm N}=\emptyset$, $\Omega$ is convex, and the PDE data are sufficiently smooth, by e.g.\ \cite[Thm.~3.12]{ern2004theory}.
	\footnote{The choice of requiring $H^2$ elementwise regularity is made only for simplicity: what we really need is that the trace of $\nabla u$ is in $L^2(\partial E)^d$ for all elements, which is ensured by $u\in V_*:=H^1(\Omega)\cap H^{\frac32+\epsilon}(\calT_h)$ for some $\epsilon>0$.}
	
	We consider the following discretization of problem \eqref{eq:BVP}:  
	\begin{equation}\label{variational}
	\text{Find } u_h \in V_h \text{ such that }
	\calA_h^\mathrm{dar}(u_h,v_h)=L_h(v_h) \quad \forall v_h \in V_h,
	\end{equation}
	with the DG bilinear form 
	$\calA_h^{\mathrm{dar}}: V_{*h} \times V_h \to \mathbb{R} $,
	\begin{align*}
	\calA_h^\mathrm{dar}(w,v_h):=&\ \calA_h^{\mathrm d}(w,v_h)+\calA_h^\mathrm{ar}(w,v_h),
	\\
	\calA_h^{\mathrm d}(w,v_h):=&
	\sum_{E\in \calT_h}	\int_E \bK  \nabla w \cdot \nabla  v_h\\&
	+
	\sum_{F\in\calF_h^{\mathrm I}}
	\int_F \Big(-\mvl{\bK   \nabla w } \cdot \jmp{v_h}- \jmp{w} \cdot \mvl{\bK   \nabla v_h}  
	+ \gamma\frac{K_F}{h_F} \jmp{w}\cdot \jmp{v_h}\Big)\\&
	+
	\sum_{F\in\calF_h^{\mathrm D}}
	\int_F \Big(-\bK   \nabla w \cdot \bn v_h- w \bK   \nabla v_h \cdot \bn+ \gamma\frac{K_F}{h_F} w v_h\Big),
	\\
	\calA_h^\mathrm{ar}(w,v_h):=&
	\sum_{E\in\calT_h}
	\int_E  \Big(-(\bbeta w) \cdot  \nabla  v_h + \sigma w  v_h\Big)\\ 
	&+
	\sum_{F\in\mathcal{F}_h^{\mathrm I}}
	\int_F \Big( \mvl{\bbeta  w} \cdot \jmp{v_h}+\frac12|\bbeta\cdot \bn_F|\jmp{w}\cdot\jmp{v_h}\Big)+
	\sum_{F\in\mathcal{F}_h^+}  
	\int_F ( \bbeta  w) \cdot \bn v_h,
	\end{align*}
	and the linear form $L_h:V_h\to\IR$,
	\begin{equation*}
	L_h(v_h):=\sum_{E\in\calT_h}\int_E  f v_h - \sum_{F\in\mathcal{F}_h^\mathrm{N}} 
	\int_F  g_{\mathrm N} v_h +
	\sum_{F\in\mathcal{F}_h^{\mathrm D}}	
	\int_F g_{\mathrm D}  \Big(- \bK   \nabla v_h  \cdot \bn  + \gamma\frac{K_F}{h_F} v_h\Big)
	-\sum_{F\in\mathcal{F}_h^-}\int_F g_D\bbeta\cdot\bn v_h.
	\end{equation*}
	The bilinear form $\calA_h^\mathrm{dar}$ depends on the penalty parameters $\gamma,K_F>0$ that penalize the jumps of the function values.
	The quantity $\gamma>0$ is a dimensionless constant independent of the diffusion coefficient $\bK$, while $K_F$ is a diffusion-dependent penalty parameter defined on each facet such that $\kmin \leq K_F\leq \N{\bK}_{L^{\infty}(E_1\cup E_2)}$ for all $F\in \calF_h^{\mathrm I}$ with $F=\partial E_1 \cap \partial E_2$, and $\kmin \leq K_F\leq \N{\bK}_{L^{\infty}(E)}$ for all $F\in \calF_h^{\mathrm B}$ with $F\subset\partial E \cap \Gamma$.
	
	Problem \eqref{variational} is independent of the choice of the normal $\bn_F$ on the internal facets, since its only occurrence in $\calA_h^\mathrm{dar}$ is inside the absolute value.
	
	The term on the interior facets in $\calA_h^\mathrm{ar}$ is the penalization form of the classical upwind flux, \cite[eq.~(20)]{brezzi2004discontinuous}.
	Indeed, if $\bx\mapsto\bbeta(\bx) \cdot \bn_F(\bx)$ does not change sign in any given $F\in\calF_h^{\mathrm I}$, then, for $F=\partial E_1 \cap \partial E_2$ and $\varphi\in H^1(\calT_h)$,
	\begin{align*}
	\mvl{\bbeta \varphi} \cdot \bn_F +\frac12\abs{\bbeta\cdot\bn_F} \jmp{\varphi}\cdot \bn_F
	=\mvl{\bbeta \varphi}_{\mathrm{upw}} \cdot \bn_F
	\quad\text{where}\;
	\mvl{\bbeta\varphi}_{\mathrm{upw}}: = \begin{cases}
	\bbeta\varphi_{|_{E_2}} & \text{if }\bbeta \cdot \bn_{E_1}<0,\\
	\bbeta\varphi_{|_{E_1}} & \text{if }\bbeta \cdot \bn_{E_1}>0,\\ 
	\bbeta	\mvl{ \varphi} & \text{if }\bbeta \cdot \bn_{E_1}=0.
	\end{cases}
	\end{align*}
	
	\begin{rem}
	The diffusion part of the DG formulation \eqref{variational} corresponds to the formulation in \cite[eq.~(2.24)]{riviere2008discontinuous} with $\alpha=0$ (no reaction), $\epsilon=-1$ (SIPG), and $\sigma_e^1=0$ for all facets (no gradient jump stabilization term).
	The penalty term is slightly different: on each facet, \cite{riviere2008discontinuous} uses a number divided by a power of the $(d-1)$-dimensional measure of the facet, while we use a constant $\gamma$ (independent of the facet) times a diffusion-dependent penalty parameter $K_F$, divided by the facet's diameter $h_F$, following \cite[eq.~(4.64)]{di2011mathematical}.
	In turn, \cite[eq.~(4.64)]{di2011mathematical} assumes piecewise-constant diffusion, uses diffusion-dependent weights for the average and $K_F$ is chosen as the harmonic mean of (scalar) $\bK$ across $F$. 
	This penalty strategy is particularly important in the advection-dominated/reaction-dominated regimes to tune automatically the penalty parameter and reduce spurious oscillations, see \cite[p.~150]{di2011mathematical} and section~\ref{ex3} below.
	
	For what concerns the advection--reaction terms, \eqref{variational} follows \cite[eq.~(2.36)]{di2011mathematical} with $\eta=1$ and with the right-hand side as in  \cite[Remark 2.17]{di2011mathematical}.
	\end{rem}

	\subsection{Mesh-dependent norms}\label{s:Norms}	
	For all $v\in V_{*h}$ we define four mesh-dependent norms and the seminorm $|\cdot|_{\mathrm J}$:
	\begin{align}\nonumber
	\vertiii{v}^2_{\mathrm d}:=&\sum_{E\in \calT_h}\int_E\bK\nabla v\cdot \nabla v+\abs{v}_{\mathrm J}^2, \qquad 
	\abs{v}_{\mathrm J}^2:=\sum_{F\in\calF_h^{\mathrm I}}\gamma\frac{K_F}{h_F}\int_F \jmp{v}^2+\sum_{F\in\calF_h^{\mathrm D}}\gamma\frac{K_F}{h_F}\int_Fv^2,\\
	\vertiii{v}^2_\mathrm{ar}:=&\sigma_0\N{v}^2_{L^{2}(\Omega)}+ \frac12\sum_{F\in\calF_h}\int_F\abs{\bbeta\cdot\bn_F}\jmp{v}^2,\label{eq:Norms}\\
	\vertiii{v}^2_\mathrm{dar}:=&	\vertiii{v}^2_{\mathrm d}+	\vertiii{v}^2_\mathrm{ar}, \nonumber\\
	\vertiii{v}^2_{\mathrm{dar},*}:=&	\vertiii{v}^2_\mathrm{dar}+\sum_{E\in \calT_h}h_E\N{\bK^{\frac12}\nabla v\cdot \bn_E}^2_{L^2(\partial E)}+\sum_{E\in \calT_h}\N{\bbeta}_{L^{\infty}(E)}\N{v}^2_{L^2(\partial E)}.
	\nonumber
	\end{align}
	We write $\bK^\frac12$ for the unique positive-definite matrix field such that $\bK^\frac12\bK^\frac12=\bK$ in $\Omega$.
	Note that $\vertiii{\cdot}_{\mathrm d}$ is a norm because we have assumed that $\Gamma_{\mathrm D}$ is not empty.
	
	\subsection{Well-posedness, stability, quasi-optimality}\label{s:WellP}
	
	The aim of this section is to prove the well-posedness of the discrete DG problem \eqref{variational} and the quasi-optimality error estimates of the DG method.
	The proof of the next theorem relies on Lax--Milgram theorem and consists of verifying the three assumptions of the abstract result in \cite[Thm.~1.35]{di2011mathematical}:
	consistency, discrete coercivity and boundedness.
	
	\Cref{th:errorestimate} holds for arbitrary polynomial spaces $V_h\subset\IP^p(\calT_h)$ (more generally, for any discrete space for which an inverse trace inequality such as \eqref{discretetraceinequality} holds).
	With this generality, we cannot immediately apply standard results such as those in \cite{di2011mathematical}: their analysis of the advection--reaction bilinear form relies on the ``boundedness on orthogonal subscales'' \cite[Lemma~2.30]{di2011mathematical}, whose proof requires that (piecewise) partial derivatives of elements of $V_h$ belong to $V_h$, a property satisfied by $\IP^p(\calT_h)$ but not all its subspaces.
	Similar assumptions are common in the literature, e.g.~\cite[eq.~(3.6)]{houston2002discontinuous}. Many works also assume piecewise-constant diffusion, e.g.~\cite[eq.~(4.3)]{houston2002discontinuous}, \cite[Assumption 4.43]{di2011mathematical}, \cite{ayuso2009discontinuous}, while we are interested in the general case $\bK\in [L^{\infty}(\Omega)]^{d\times d}$.
	These hypothesis are not necessary and are often made for simplicity of presentation, however we can not directly rely on their analysis.
	We refer to \cite[sect.~5.1--5.2]{CangianiDGH2017} for a more general analysis of an inconsistent variant of the SIP-upwind DG method for second-order PDEs with nonnegative characteristic form.
	
	\Cref{th:errorestimate} gives an explicit estimate, which in \cref{s:hConvergence} will be combined with
	the local approximation bound \eqref{eq:Approx} of the quasi-Trefftz space.
	In particular, our analysis for the discrete coercivity of the diffusion bilinear form follows \cite[sect.~2.7.1]{riviere2008discontinuous}, while the continuity is similar to \cite[Lem\-ma~4.52]{di2011mathematical}. 
	Concerning the advection--reaction bilinear form, for the coercivity we follow \cite{brezzi2004discontinuous}, while, 
	to prove continuity avoiding
	conditions like \cite[eq.~(3.6)]{houston2002discontinuous}, we estimate the quantity
	$\abs{\sum_{E\in\calT_h}	\int_E  (\bbeta v) \cdot  \nabla  w_h}$  
	using the diffusion norm $\vertiii{w_h}_{\mathrm{d}}$ for the second term.

	\begin{theorem}\label{th:errorestimate}
	Under the assumptions on the BVP and the mesh made in \cref{s:ProblemDAR,s:Mesh}, let 
	$\gamma_0:=\frac{\N{\bK}^2_{L^{\infty}(\Omega)}}{\kmin ^2}N_{\partial} \frac{(p+1)(p+d)}{ r_\star}> 0$ with
	 $r_\star$ defined in {\rm\ref{it:starshaped}}, $N_{\partial}$ in \eqref{Npartial} and $\kmin $ in \eqref{ellipticity}, and recall $\sigma_0$ from \eqref{assumptiondiv}. 
	Assume that the penalty parameter satisfies $\gamma>\gamma_0$, and set 
	$$\alpha:=1-\sqrt{\frac{\gamma_0}{\gamma}},\qquad
	M:=5+ \frac{\N{\bbeta}_{L^{\infty}(\Omega)}} {\sqrt{\kmin  \sigma_0}}+	\frac{\N{\sigma}_{L^{\infty}(\Omega)} }{\sigma_0}
	+\bigg(\frac{\N{\bK}_{L^{\infty}(\Omega)}}{\gamma \kmin}\bigg)^{\frac12}.
	$$
	Then the bilinear form $\calA_h^\mathrm{dar}$ is coercive on $V_h$ in $\vertiii{\cdot}_\mathrm{dar}$ norm:
	\begin{equation}\label{eq:coercive}
	\calA_h^\mathrm{dar}(v_h,v_h)\ge \alpha\vertiii{v_h}^2_\mathrm{dar} \qquad \forall v_h\in V_h.
	\end{equation}
	The DG variational problem \eqref{variational} admits a unique solution $u_h\in V_h$, for any subspace 
	$V_h\subset \IP^p(\calT_h)$.
	The bilinear form $\calA_h^\mathrm{dar}$ is bounded on $V_{*h}\times V_h$ in $\vertiii{\cdot}_{\mathrm{dar},*}$--$\vertiii{\cdot}_\mathrm{dar}$ norms: 
	\begin{equation*}
	\calA_h^\mathrm{dar}(v,w_h)\leq M\vertiii{v}_{\mathrm{dar},*} \vertiii{w_h}_\mathrm{dar} \qquad \forall (v,w_h)\in V_{*h}\times V_h.
	\end{equation*}
	The weak solution $u$ of the BVP \eqref{eq:BVP}
	solves the variational problem \eqref{variational}, i.e.\ \eqref{variational} is consistent.
	Moreover, the following quasi-optimality error estimate holds true:
	\begin{equation}\label{quasioptimality}
	\vertiii{u-u_h}_\mathrm{dar}\leq \left(1+\frac{M}{\alpha}\right) \inf_{v_h\in V_h}\vertiii{u-v_h}_{\mathrm{dar},*}.
	\end{equation}
	\end{theorem}
	
	\begin{proof}
	\textbf{Discrete Coercivity}: First we establish the coercivity of the diffusion bilinear form $\calA_h^{\mathrm d}$ on $V_h$ with respect to the $\N{\cdot}_{\mathrm d}$-norm, then we show that the advection--reaction bilinear form $\calA_h^\mathrm{ar}$ is coercive on $V_h$ with respect to the $\N{\cdot}_\mathrm{ar}$-norm.  Combining these two results, we deduce the discrete coercivity of the diffusion--advection--reaction bilinear form $a^\mathrm{dar}_h$ with respect to the $\N{\cdot}_\mathrm{dar}$-norm.
	
	Let $v_h\in V_h$.
	Applying Young's inequality to the bound \eqref{consistencyterm} proved in the appendix we deduce 
	\begin{align*}
	&\Bigg|\sum_{F\in\calF_h^{\mathrm I}\cup \calF_h^{\mathrm D}} \int_F  \mvl{\bK   \nabla v_h} \cdot \jmp{v_h}\Bigg|\\
	&\leq\frac{\N{\bK}_{L^{\infty}(\Omega)}}{\kmin }
	\left( \frac{N_{\partial} (p+1)(p+d)}{\gamma\,r_\star}\right)^{\frac12}
	\left(   \frac12	
	\sum_{E\in\calT_h} 
	\N{\bK^{\frac12}\nabla v_h}_{L^2( E)}^2
	+  \frac12	 \abs{v_h}_{\mathrm J}^2\right).
	\end{align*}
	Using this bound we achieve
	$\calA_h^{\mathrm d}(v_h,v_h)\ge
	\big(1-	\frac{\N{\bK}_{L^{\infty}(\Omega)}}{\kmin}
	( \frac{N_{\partial} (p+1)(p+d)}{\gamma\, r_\star})^{\frac12}\big)\vertiii{v_h}_{\mathrm d}^2$.
	Choosing $\gamma$ large enough, $\gamma > \gamma_0=  \frac{\N{\bK}^2_{L^{\infty}(\Omega)}}{k_{\min}^2}
	N_{\partial} \frac{(p+1)(p+d)}{ r_\star}$, we obtain the discrete coercivity 
	$\calA_h^{\mathrm d}(v_h,v_h)\ge (1-\sqrt{\frac{\gamma_0}\gamma})\vertiii{v_h}_{\mathrm d}^2$
	of the diffusion bilinear form.
	
	On the other hand, integration by parts yields
	\begin{equation*}
	\sum_{E\in\calT_h}\int_E  (\bbeta v_h) \cdot  \nabla  v_h=\sum_{E\in\calT_h}\int_E \bbeta \cdot  \nabla  \left(\frac{v_h^2}{2}\right)
	=-\sum_{E\in\calT_h}\int_E  \mathrm{div}(\bbeta) \frac{v_h^2}{2} +\sum_{E\in\calT_h}\int_{\partial E}  \bbeta \cdot \bn_E \frac{v_h^2}{2}.
	\end{equation*}
	Applying the DG magic formula \eqref{DG_magic} on the last term with $\bw=\bbeta$ and $\varphi=v_h^2$, using the formula 
	$\frac12\mvl{\bbeta} \cdot \jmp{v_h^2}=\mvl{\bbeta v_h} \cdot \jmp{v_h}-\frac14\jmp\bbeta|\jmp{v_h}|^2$
	on each interior facet, and observing that $\jmp{\bbeta}=0$ on $F\in\calF_h^{\mathrm I}$ by the regularity assumption \eqref{assumptionbeta}, we get
	\begin{equation*}
	\sum_{E\in\calT_h}\int_{\partial E}  \bbeta \cdot \bn_E \frac{v_h^2}{2}=
	\sum_{F\in\calF_h^{\mathrm I}}\int_F\mvl{\bbeta v_h} \cdot \jmp{v_h}+\frac12\sum_{F\in\calF_h^{\mathrm B}}\int_F\bbeta \cdot\bn v_h^2.
	\end{equation*}
	Combining the previous steps, the bilinear form $\calA_h^\mathrm{ar}(v_h,v_h)$ can be rewritten as follows:
	\begin{align*}
	\sum_{E\in\calT_h}	\int_E  \bigg(\sigma+\frac{\mathrm{div}\bbeta}{2}\bigg) v_h^2 
	-\frac12\sum_{F\in\calF_h^-}\int_F\bbeta \cdot\bn v_h^2
	+\frac12\sum_{F\in\calF_h^+}\int_F\bbeta \cdot\bn v_h^2+
	\frac12 \sum_{F\in\calF_h^{\mathrm I}}	\int_F  |\bbeta\cdot \bn_F|\jmp{v_h}^2.
	\end{align*}
	Recalling the definition \eqref{inflowoutflow} of $\Gamma_\pm$ and using the lower bound \eqref{assumptiondiv} on $\sigma+\frac12\mathrm{div}\bbeta$,
	we deduce that
	\begin{equation*}
	\calA_h^{\mathrm{ar}}(v_h,v_h)\ge
	\sigma_0\N{v_h}_{L^{2}(\Omega)}^2+
	\frac12 \sum_{F\in\calF_h}	\int_F  |\bbeta\cdot \bn_F|\jmp{v_h}^2
	=\vertiii{v_h}_\mathrm{ar}^2,
	\end{equation*}
	hence the coercivity constant for the advection--reaction bilinear form is $1$.  
	Since $\calA^\mathrm{dar}_h=\calA^{\mathrm d}_h+\calA^\mathrm{ar}_h$ and 
	$\vertiii{\cdot}_\mathrm{dar}^2=\vertiii{\cdot}_{\mathrm d}^2+\vertiii{\cdot}_\mathrm{ar}^2$, we obtain the discrete coercivity \eqref{eq:coercive}.
	The discrete coercivity implies the well-posedness of the discrete DG problem  \eqref{variational} since it is a sufficient condition for discrete stability \cite[Lemma~1.30]{di2011mathematical}.

	\textbf{Consistency}:
	Let $u\in V^*$ be the weak solution of problem \eqref{eq:BVP}. 
	We show that $u$ satisfies the variational problem \eqref{variational}, i.e.
	$
	\calA_h^{\mathrm{dar}}(u,v_h)=L_h(v_h)$ for all $v_h \in V_h.
	$
	We multiply \eqref{eq} by $v_h\in V_h$, integrate by parts on each element $E$ and sum over all the elements:
	\begin{equation*}
	-\sum_{E\in\calT_h}	\int_E \left(-\bK  \nabla u+\bbeta u\right) \cdot \nabla  v_h +
	\sum_{E\in\calT_h}	\int_{\partial E}   \left(-\bK  \nabla u+\bbeta u\right) \cdot \bn_E v_h +
	\sum_{E\in\calT_h}	\int_E  \sigma u  v_h= 	\sum_{E\in\calT_h}\int_E f  v_h.
	\end{equation*}
	Using the DG magic formula \eqref{DG_magic}  with $\bw=-\bK\nabla u+\bbeta u$ and $\varphi=v_h$ on the second term and observing that $\jmp{-\bK\nabla u+\bbeta u} =0 $ on each interior facet since $	-\bK \nabla u +\bbeta u $ belongs to $H(\mathrm{div};\Omega)$, and using the Dirichlet and Neumann boundary conditions \eqref{Dirichlet}--\eqref{Neumann}, we find
	\begin{align*}
	&	-\sum_{E\in\calT_h}	\int_E \left(-\bK  \nabla u+\bbeta u\right) \cdot \nabla  v_h -
	\sum_{F\in\calF_h^{\mathrm I}}	\int_F  \mvl{ \bK   \nabla u } \cdot \jmp{v_h} +
	\sum_{F\in\calF_h^{\mathrm I}\cup\calF_h^+} 
	\int_F  \mvl{\bbeta  u} \cdot \jmp{v_h}	\\&
	-	\sum_{F\in\calF_h^{\mathrm D}}	\int_F  \bK   \nabla u \cdot \bn v_h+\sum_{E\in\calT_h}	\int_E  \sigma u  v_h= \int_{\Omega} f  v_h - \sum_{F\in\calF_h^\mathrm{N}}	\int_F g_{\mathrm N} v_h   
	- \sum_{F\in\calF_h^-}	\int_F  ( \bbeta  g_{\mathrm D})\cdot \bn v_h .
	\end{align*}
	Using the fact that $\jmp{u}=0$ on each interior facet since $u\in H^{1}(\Omega)$, and that $u$ satisfies the Dirichlet boundary condition \eqref{Dirichlet}, the variational formulation \eqref{variational} evaluated in $u$ coincides with the above equality, 
	implying the consistency of the DG scheme.
	
	\textbf{Boundedness}:
	Let $(v,w_h)\in V_{*h}\times V_h$. We decompose the bilinear form $\calA_h^\mathrm{dar}$ in eight terms:
	\begin{align*}
	\calA_h^\mathrm{dar}(v,w_h)=&
	\sum_{E\in\calT_h}	\int_E \bK  \nabla v \cdot \nabla  w_h +\sum_{F\in\calF_h^{\mathrm I}\cup\calF_h^{\mathrm D}}	\gamma\frac{K_F}{h_F}\int_F  \jmp{v}\cdot \jmp{w_h}-
	\sum_{F\in\calF_h^{\mathrm I}\cup\calF_h^{\mathrm D}}	\int_F  \mvl{\bK\nabla v } \cdot \jmp{w_h}  \\&
	-\sum_{F\in\calF_h^{\mathrm I}\cup\calF_h^{\mathrm D}}	\int_F  \jmp{v}\cdot  \mvl{ \bK   \nabla w_h }
	+\sum_{E\in\calT_h}	\int_E  (-(\bbeta v) \cdot  \nabla  w_h + \sigma  v w_h)\\
	&+\sum_{F\in\calF_h^{\mathrm I}}	\int_F  \mvl{\bbeta  v} \cdot \jmp{w_h}+
	\frac12 \sum_{F\in\calF_h^{\mathrm I}}	\int_F  |\bbeta\cdot \bn_F|\jmp{v}\cdot\jmp{w_h}  +
	\sum_{F\in\calF_h^+}
	\int_F ( \bbeta  v) \cdot \bn_F w_h\\
	=:&\mathfrak{T}_1+\mathfrak{T}_2+\mathfrak{T}_3+\mathfrak{T}_4+\mathfrak{T}_5+\mathfrak{T}_6+\mathfrak{T}_7+\mathfrak{T}_8.
	\end{align*}
	The Cauchy--Schwarz inequality and the ellipticity condition \eqref{ellipticity} yield
	\begin{align*}
	\abs{\mathfrak{T}_1 + \mathfrak{T}_2} 
	\leq&\ \vertiii{v}_{\mathrm d} 	\vertiii{w_h}_{\mathrm d},\\
	\abs{\mathfrak{T}_7 + \mathfrak{T}_8} 
	\leq&\ 2 \vertiii{v}_\mathrm{ar}\vertiii{ w_h}_\mathrm{ar},\\
	\abs{\mathfrak{T}_5} 
	\leq&\ \frac{\N{\bbeta}_{L^{\infty}(\Omega)}} {\sqrt{\kmin  \sigma_0}}\vertiii{v}_\mathrm{ar}
	\vertiii{w_h}_{\mathrm d}+
	\frac{\N{\sigma}_{L^{\infty}(\Omega)} }{\sigma_0}\vertiii{v}_\mathrm{ar}\vertiii{ w_h}_\mathrm{ar}.
	\end{align*}
	Moreover, using the continuity of $\bbeta$ \eqref{assumptionbeta} and the Cauchy--Schwarz inequality, we infer
	\begin{align*}
	\abs{\mathfrak{T}_6} &\leq  
	\bigg(2\sum_{F\in\calF_h^{\mathrm I}}\int_F \abs{\bbeta\cdot\bn_F}\mvl{v}^2\bigg)^{\frac12}\ 
	\bigg(\frac12 \sum_{F\in\calF_h^{\mathrm I}}\int_F \abs{\bbeta\cdot\bn_F} \jmp{w_h}^2\bigg)^{\frac12}\\
	&\leq  \bigg(  \sum_{E\in\calT_h} \N{\bbeta}_{L^{\infty}(E)}\N{v}^2_{L^2(\partial E)} \bigg)^{\frac12}\  \vertiii{ w_h}_\mathrm{ar}
	\leq 
	\vertiii{v}_{\mathrm{dar},*}\vertiii{ w_h}_\mathrm{ar},        
	\end{align*}
	where in the second step we use the formula $2\mvl{v}^2=\frac12(v_1+v_2)^2\leq v_1^2+v_2^2$. 
	
	Since $h_F\leq h_E$ for all $F\in\calF_{E}$, $E\in\calT_h$, and $\kmin \leq K_F$ for all $F\in\calF_h$, from to the bound \eqref{consistencytermbound} we get
	$\abs{\mathfrak{T}_3}\leq
	\big(\frac{\N{\bK}_{L^{\infty}(\Omega)}}{\gamma\kmin}\big)^{\frac12}
	\vertiii{v}_{\mathrm{dar},*}\vertiii{w_h}_{\mathrm d}.$
	Finally, we control the remaining term using bound~\eqref{consistencyterm}:
	$\abs{\mathfrak{T}_4}\leq\big(\frac{\gamma_0}{\gamma}\big)^\frac12\vertiii{v}_{\mathrm d}\vertiii{w_h}_{\mathrm d}
	\leq\vertiii{v}_{\mathrm d}\vertiii{w_h}_{\mathrm d}$.
	By combining all these bounds we infer the boundedness of $\calA_h^\mathrm{dar}$ with 
	$M$ as in the statement.
	
	Since discrete stability, consistency and boundedness hold, we conclude applying 
	\cite[Thm.~1.35]{di2011mathematical}.
	\end{proof}
	
	Given the quasi-optimality inequality \eqref{quasioptimality}, the convergence of the DG method follows studying
	the approximation properties of the particular discrete space $V_h$ chosen.
	In \Cref{th:finalerror} we do this for the $h$-convergence of the quasi-Trefftz version of the DG scheme.
	
	Among all the constants and the parameters appearing in \Cref{th:errorestimate}, only the maximal number of facets per element $N_\partial$ \eqref{Npartial} and 
	the star-shapedness parameter $r_\star$ \ref{it:starshaped} depend on the mesh $\calT_h$, and both are
	easily computed. 
	The dependence on the polynomial degree $p$ is explicit.
	\begin{rem}\label{rem::Peclet}
    The quasi-optimal estimate \eqref{quasioptimality} is not entirely satisfactory because the continuity constant $M$ has an unfavorable dependence on the dimensionless quantity  $\frac{\N{\bbeta}_{L^{\infty}(\Omega)}} {\sqrt{\kmin  \sigma_0}}$,
    which may lead to a non-robust error bound in the advection-dominated regime. 
	In the standard DG analysis, robustness is achieved using the ``boundedness on orthogonal subscales'' \cite[\S2.3.2 and \S4.6.3]{di2011mathematical} for the treatment of the term $\mathfrak{T_5}$, while in this setting we cannot rely on this argument since the (piecewise) partial derivatives of a quasi-Trefftz function of degree $p$ does not necessarily belong to the quasi-Trefftz space of the same degree.
	However, numerically we do not observe a significant difference between the standard DG method and the quasi-Trefftz DG method as the problem becomes increasingly advection-dominated, implying that our estimate is likely to be non-sharp in this limit, see Figure \ref{fig:nu}.
	\end{rem}

	\section{Quasi-Trefftz DG discretization}\label{s:QTDGdiscretizazion}
	We fix a point $\bx^E\in E$ for each mesh element $E\in\calT_h$.
	Since the diffusion--advection--reaction operator $\calL$, defined in \eqref{daroperator}, is a linear partial differential operator of order $m=2$, the quasi-Trefftz space \eqref{inhomQT} for the equation $\calL u = f$ on a mesh element $E\in\calT_h$ is 
	\begin{equation}\label{inhomQT_dar}
	\QT^p_f(E)=\big\{ v\in \IP^p(E) \mid D^{\bi} \mathcal{L}v (\bx^E)=D^{\bi} f (\bx^E) \quad
	\forall \bi\in \IN_0^d, \;|\bi|\leq p-2\big\}, \quad  p\in \IN.
	\end{equation}
	For $p=1$ we have $\QT^1_f(E)=\IP^1(E)$, so we fix $p\ge2$.
	Recall \eqref{assalg}: the non-degeneracy condition \eqref{hpcoeff} is ensured by ellipticity \eqref{ellipticity}.
	The space $\QT^p_f(E)$ is well-defined if the PDE coefficients $\bK$, $\bbeta$ and $\sigma$ and the source term $f$ are sufficiently smooth. 
	We expand the operator $\calL v=\mathrm{div}(-\bK\nabla v+\bbeta v)+\sigma v$ in the form \eqref{linearop} using the Leibniz product rule:
	\begin{align*}
	\mathcal{L}v =\sum_{j=1}^{d} \bigg[\sum_{m=1}^{d}  \big(-\bK_{jm}  D^{\be_{j}+\be_{m}} v-D^{\be_{j}}\bK_{jm}D^{\be_{m}} v \big) +	
	\bbeta_{j} 	 D^{\be_{j}} v+(D^{\be_j}\bbeta_j) v \bigg]+\sigma v.
	\end{align*}
	Recalling the regularity hypothesis \eqref{eq:Cpm} made for general differential operators, assume
	\begin{equation}\label{eq:RegularityQTDG}
	\bK\in C^{p-2} (E)^{d\times d}, \qquad
	\bm{\mathrm{div}}\bK,\ 
	\bbeta\in C^{p-2}(E)^d,\qquad
	\mathrm{div}\bbeta,\ \sigma,\ f\in C^{p-2}(E).
	\end{equation}
	where the matrix divergence $\bm{\mathrm{div}}\bK$ is taken column-wise.
	Then the quasi-Trefftz space \eqref{inhomQT_dar} for the diffusion--advection--reaction equation is well-defined and all the results in \cref{s:QT} apply.
	The detailed description of \Cref{Algo:general} for the homogeneous diffusion--advection--reaction equation, for the case $d=1$, $d=2$, and for the general $d$-dimensional case, can be found in \cite[sec.~5.5]{perinati2023quasitrefftz}.
	
	We discretize the DG formulation \eqref{variational} choosing as trial space the global quasi-Trefftz space 
	$\QT^p_f(\calT_h):=\{v\in L^2(\Omega) \mid v_{|_T} \in \QT^p_f(E)  \; \forall E \in \calT_h \}$ and as test space the global quasi-Trefftz space 
	$\QT^p_0(\calT_h):=\{v\in L^2(\Omega) \mid v_{|_T} \in \QT^p_0(E)  \; \forall E \in \calT_h \}$.
	The quasi-Trefftz DG method is then: 
	\begin{align}\label{qtvariational}
	\text{Find } u_h \in \QT^p_f(\calT_h) \text{ such that }\quad \calA_h^{\mathrm{dar}}(u_h,v_h)=L_h(v_h) 
	\qquad \forall v_h \in \QT^p_0(\calT_h).
	\end{align}
	If the source term $f$ vanishes, existence and uniqueness of $u_h$ in \eqref{qtvariational} follow from 
	\Cref{th:errorestimate}.
	However, in the general case with $f\neq 0$, 
	\Cref{th:errorestimate} does not apply directly,
	since the trial and test spaces are different.
	In this case, we choose a lifting $u_{h,f}\in\QT^p_f(\calT_h)$.
	This can be computed by applying \Cref{Algo:general} in each element $E\in\calT_h$, with any choice of Cauchy data $(\psi_0,\psi_1)\in\IP^p(\IR^{d-1})\times\IP^{p-1}(\IR^{d-1})$.
	For simplicity, in the 
	experiments of \cref{s:numexp} we take $\psi_0=\psi_1=0$.
	Then we consider the problem
	\begin{align}\label{qtvariational2}
	\begin{split}
    & \text{Find } u_{h,0} \in \QT^p_0(\calT_h) \text{ s.t. } \quad
	 \calA_h^{\mathrm{dar}}(u_{h,0},v_h)=L_h(v_h)-\calA_h^{\mathrm{dar}}(u_{h,f},v_h) \qquad \forall v_h \in \QT^p_0(\calT_h).
	\end{split}
	\end{align}
	Here trial and test spaces coincide and \Cref{th:errorestimate} applies, so problem \eqref{qtvariational2} admits a unique solution $u_{h,0}$.
	Then $u_h=u_{h,0}+u_{h,f}$ is the solution to \eqref{qtvariational}.

	\subsection{\texorpdfstring{$h$}{h}-convergence of the quasi-Trefftz DG method}
	\label{s:hConvergence}
	
	The aim of this section is to infer the convergence rate in $h$ for the quasi-Trefftz Galerkin error $u-u_h$ measured in the $\N{\cdot}_\mathrm{dar}$-norm.	
	We first adapt the DG stability analysis of \cref{s:WellP} to problem \eqref{qtvariational}, which is posed on an affine trial space.
	
	\begin{theorem}\label{th:qtdg-wellposed}
	Under the assumptions of \Cref{th:errorestimate}
	and \eqref{eq:RegularityQTDG}, problem \eqref{qtvariational} is well-posed and 
	the following error estimate holds true:
	\begin{equation}\label{qtquasioptimality}
	\vertiii{u-u_h}_\mathrm{dar}\leq \left(1+\frac{M}{\alpha}\right) \inf_{v_h\in \QT^p_f(\mathcal T_h)}\vertiii{u-v_h}_{\mathrm{dar},*}.
	\end{equation}
	\end{theorem}
	\begin{proof}
	Under the assumptions made, \Cref{prop:Uniqueness} ensures the existence of $u_{h,f}\in\QT_f^p(\calT_h)$.
	\Cref{th:errorestimate} implies the existence of $u_{h,0}$ solving \eqref{qtvariational2}, and so of $u_h=u_{h,0}+u_{h,f}$ solving \eqref{qtvariational}.
	The uniqueness of $u_h$ follows because \eqref{qtvariational} is a square discrete linear problem as 
	$\dim(\QT^p_f(\calT_h))=\dim(\QT^p_0(\calT_h))$.

	To show \eqref{qtquasioptimality}, we adapt C\'ea lemma to the affine space in \eqref{qtvariational}.
	For any $v_h\in \QT^p_f(\mathcal T_h)$, $u_h-v_h\in \QT^p_0(\mathcal T_h)$ and therefore 
	$\calA_h^\mathrm{dar}(u-u_h, u_h-v_h)=0$ by \eqref{qtvariational} and the consistency of the scheme.
	Coercivity and continuity of the DG formulation yield the estimate
	\begin{align*}  
	\alpha \vertiii{u_h-v_h}_\mathrm{dar}^2 &\leq \calA_h^\mathrm{dar}(u_h-v_h,u_h-v_h)\\
	&=\calA^\mathrm{dar}_h(u-v_h,u_h-v_h)
	\leq M\vertiii{u-v_h}_{\mathrm{dar},*} \vertiii{u_h-v_h}_\mathrm{dar} \qquad \forall v_h\in\QT^p_f(\calT_h). 
	\end{align*}
	Estimate \eqref{qtquasioptimality} follows by applying the triangle inequality and recalling that 
	$\vertiii{\cdot}_\mathrm{dar}\le\vertiii{\cdot}_{\mathrm{dar},*}$.
	\end{proof}

	From this quasi-optimality result we deduce the optimal convergence rate for the quasi-Trefftz DG method, using the approximation estimate \eqref{eq:Approx}.
	We define the broken space $C^{q}(\calT_h):=\{v \in L^{2}(\Omega) \mid v_{|_E} \in C^{q}(E)  \quad \forall E \in \calT_h \}$ for $q\in\IN_0$ and recall that $V_{*} :=H^1(\Omega) \cap H^2(\calT_h)$ 
	
	\begin{theorem}[Quasi-Trefftz DG convergence rate]\label{th:finalerror}
	Let $p\in \IN$ and let $u\in V_{*}\cap C^{p+1}(\calT_h)$ solve the BVP \eqref{eq:BVP} under the assumptions made in \cref{s:ProblemDAR} and \eqref{eq:RegularityQTDG}.
	Let $u_h$ solve \eqref{qtvariational}, with a mesh $\calT_h$ as in \cref{s:Mesh}, and
	penalty parameter $\gamma$ as in \Cref{th:errorestimate}.
	Assume that each mesh element $E$ is star-shaped with respect to $\bx^E$.
	Then, the following error bound holds:
	\begin{align}\label{eq:finalerror}
	\vertiii{u-u_h}_\mathrm{dar}\leq&\ 
	\bigg(1+\frac{M}{\alpha}\bigg)\frac{d^p}{p!}
	\bigg(\sum_{E\in \calT_h} G_E \abs{E}h_E^{2p}\abs{u}^2_{C^{p+1}(E)}\bigg)^{\frac12}\\
	\le &\ 
	\bigg(1+\frac{M}{\alpha}\bigg)\frac{d^p}{p!}|\Omega|^\frac12 \; h^p \; 
	\max_{E\in\calT_h} \Big(G_E^\frac12\abs{u}_{C^{p+1}(E)}\Big)
	,\qquad \text{where}
	\nonumber\\
	G_E:=\biggl[\bigg(1+&\frac{d}{r_\star}\bigg)\N{\bK}_{L^{\infty}(E)} 
	+\frac{d^{2}}{(p+1)^2}\left(2\Cg  \gamma \N{\bK}_{L^{\infty}(\calP_E)}\frac{d}{r_\star}
	+2\N{\bbeta}_{L^{\infty}(E)} \frac{d}{r_\star} h_E+\sigma_0 h_E^2
	\right)\biggr],
	\nonumber
	\end{align}
	with $\alpha$ and $M$ as in \Cref{th:errorestimate}, $r_\star,\Cg$ in \ref{it:starshaped}--\ref{it:Graded}, $\sigma_0$ in \eqref{assumptiondiv} and 
	$\calP_E:=E\cup\bigcup_{F=\partial E\cap\partial E'\in\calF_h^{\mathrm I}}E'$  
	the patch of mesh elements adjacents to the element $E$.
	\end{theorem}
	\begin{proof}     
	We estimate the quantity
	$\inf_{v_h\in \QT^p_f(\calT_h)}\vertiii{u-v_h}_{\mathrm{dar},*}$ on the right-hand side of the quasi-optimality inequality \eqref{qtquasioptimality}, with the $\vertiii{\cdot}_{\mathrm{dar},*}$-norm defined in \cref{s:Norms} for  $v\in V_{*h}=V_*+\QT^p_f(\calT_h)$.
	We use $\jmp{v}^2=(v_1-v_2)^2\leq 2(v_{1}^2+v_2^2)$ on internal facets $F=\partial E_1\cap \partial E_2$,
	and $\jmp{v}^2=v^2$ on boundary facets $F$. 
	We recall $h_E\le\Cg h_F$ for $F\in\calF_E$ by the graded-mesh assumption \eqref{gradedmesh}, and that $K_F\leq \N{\bK}_{L^{\infty}(\calP_E)}$ for all facets $F\in \calF_E
	$. 
	Using these facts, we rearrange the sums over parts of the mesh skeleton as sums over elements and obtain the bound:
	\begin{align*}
	\vertiii{v}^2_{\mathrm{dar},*}\leq
	\sum_{E\in \calT_h} \bigg(&\N{\bK}_{L^{\infty}(E)} \N{\nabla v}^2_{L^2(E)} +2 \Cg \frac{\gamma}{h_E} \N{\bK}_{L^{\infty}(\calP_E)}\N{v}^2_{L^2(\partial E)}
	+\sigma_0\N{v}^2_{L^2(E)}
	\\&
	+\N{\bK}_{L^{\infty}(E)}h_E\N{\nabla v}^2_{L^2(\partial E)}
	+2\N{\bbeta}_{L^{\infty}(E)}\N{v}^2_{L^2(\partial E)}\bigg).
	\end{align*}
	Next, we use the definition of the $\N{\cdot}_{C^{m}}$-norms and obtain
	\begin{align*}
	\vertiii{v}^2_{\mathrm{dar},*}\leq
	\sum_{E\in \calT_h}	\bigg(&\N{\bK}_{L^{\infty}(E)}\abs{E} \N{\nabla v}^2_{C^0(E)} +
	2\Cg \frac{\gamma}{h_E} \N{\bK}_{L^{\infty}(\calP_E)}\abs{\partial E}\N{v}^2_{C^0( E)}
	+\sigma_0\abs{E}\N{v}^2_{C^0(E)}\\&
	+\N{\bK}_{L^{\infty}(E)}\abs{\partial E}h_E\N{\nabla v}^2_{C^0( E)}
	+2\N{\bbeta}_{L^{\infty}(E)}\abs{\partial E}\N{v}^2_{C^0(E)}\bigg).
	\end{align*}
	Considering the quantity of interest and using the quasi-Trefftz approximation estimate \eqref{eq:Approx} we get 
	\begin{align*}
	&\inf_{v_h\in \QT^p_f(\calT_h)}\vertiii{u-v_h}^2_{\mathrm{dar},*}\\
	\leq&
	\sum_{E\in \calT_h}	\inf_{v_h\in \QT^p_f(E)}
	\biggl[\left(\abs{E}+\abs{\partial E}h_E\right)\N{\bK}_{L^{\infty}(E)} \N{\nabla (u-v_h)}^2_{C^0(E)}\\&
	\hspace{25mm}+\left(2\Big( \Cg\frac{\gamma}{h_E}\N{\bK}_{L^{\infty}(\calP_E)}
	+\N{\bbeta}_{L^{\infty}(E)} \Big)\abs{\partial E}
	+\sigma_0\abs{E}
	\right)\N{u-v_h}^2_{C^0(E)}\biggr]
	\\\leq&
	\sum_{E\in \calT_h}	
	\left[\left(\abs{E}+\abs{\partial E}h_E\right)\N{\bK}_{L^{\infty}(E)} 
	\frac{d^{2p}}{(p!)^2} h_E^{2p}\abs{u}^2_{C^{p+1}(E)} \right. \\&
	\left.\hspace{10mm}
	+\left(2\Big( \Cg 	\frac{\gamma}{h_E}\N{\bK}_{L^{\infty}(\calP_E)} +\N{\bbeta}_{L^{\infty}(E)}\Big) \abs{\partial E}+\sigma_0\abs{E}
	\right)
	\frac{d^{2(p+1)}}{((p+1)!)^2} h_E^{2(p+1)}\abs{u}^2_{C^{p+1}(E)} \right].
	\end{align*}
	By the chunkiness property $h_E|\partial E|\le\frac{d}{r_\star}|E|$ \eqref{eq:chunky} on the mesh,
	the last expression is bounded by
	\begin{align*}
	\inf_{v_h\in \QT^p_f(\calT_h)}\vertiii{u-v_h}^2_{\mathrm{dar},*}\leq&\ \frac{d^{2p}}{(p!)^2}
	\sum_{E\in \calT_h}	
	\biggl[\left(1+\frac{d}{r_\star}\right)\abs{E}\N{\bK}_{L^{\infty}(E)} 
	+\frac{d^{2}}{(p+1)^2}\abs{E}h_E^{2}\times \\&
	\left(2	\Cg  \frac{\gamma}{h_E} \N{\bK}_{L^{\infty}(\calP_E)}\frac{d}{r_\star} h_E^{-1}
	+2\N{\bbeta}_{L^{\infty}(E)} \frac{d}{r_\star} h_E^{-1}+\sigma_0\right)\biggr] h_E^{2p}\abs{u}^2_{C^{p+1}(E)}\!.
	\end{align*}
	Combining this bound  with the quasi-optimality inequality \eqref{qtquasioptimality} yields the assertion.
	\end{proof}
	
	The estimate \eqref{eq:finalerror} can immediately be adapted to the case 
	where a different polynomial degree $p_E\in\IN$ is used in each element.
	
	For the quasi-Trefftz DG error estimate \eqref{eq:finalerror} to hold, the solution $u$ needs to belong to $C^{p+1}(\calT_h)$, which is a stronger regularity assumption than the usual $u\in H^{p+1}(\calT_h)$.
	This is a consequence of the approximation estimate \eqref{eq:Approx}, which is based on a Taylor argument.
	For the Trefftz space the analysis has been extended to the case of solutions in $H^{p+1}(\calT_h)$ using the fact that the ``averaged Taylor polynomials'' of exact solutions are Trefftz functions \cite[Lemma~1]{moiola2018space}. 
	However, we cannot use this argument since, in general, averaged Taylor polynomials are not quasi-Trefftz functions and, to our knowledge, a quasi-Trefftz convergence analysis using Sobolev norms is still missing \cite[Rem.~4.7]{IGMS_MC_2021}.
	Apart from this difference, \eqref{eq:finalerror} shows optimal $h$-convergence rates in the $\N{\cdot}_\mathrm{dar}$-norm.

	\section{Numerical experiments}\label{s:numexp}
	
	The quasi-Trefftz DG method has been implemented using \texttt{NGSolve} \cite{ngsolve} and \texttt{NGSTrefftz} \cite{ngstrefftz}\footnote{Reproduction material is available in \cite{repdata}, documentation on\\
	\url{https://paulst.github.io/NGSTrefftz/notebooks/qtelliptic.html}.}.
	The computations were performed with parallelization limited to 16 threads on a server with two Intel(R) Xeon(R) CPU E5-2687W v4, with 12 cores each.
	The derivatives required for the computation of the quasi-Trefftz functions are computed using the symbolic differentiation capabilities of \texttt{NGSolve}, with evaluation of \Cref{Algo:general} performed in parallel elementwise. 
	To initialize the algorithm, we choose the Cauchy data $\psi_0=\psi_1=0$ for constructing the lifting $u_{h,f}\in \QT^p_f(\calT_h)$, and centered monomial bases of $\IP^p(\IR^{d-1})$ and of $\IP^{p-1}(\IR^{d-1})$ as Cauchy data for the quasi-Trefftz basis \eqref{eq:BEp} of $\QT^p_0(\calT_h)$.
	We use a direct solver based on the \texttt{UMFPACK} library.
	The diffusion-dependent penalty parameter $K_F$ is chosen equal to $k_{min}$ on each facet $F\in\calF_h^{\mathrm I}$. 
	Additional experiments and details on a 2D Matlab implementation for the homogeneous case can be found in \cite{perinati2023quasitrefftz}.
	
	\subsection{Non-homogeneous Dirichlet problem}\label{ex2}
	We consider a non-homogeneous diffusion-dominated problem in the unit cube $\Omega=(0,1)^3$.
	The PDE coefficients and the solution are chosen as
	\begin{align}\label{eq:ex2}
	\bK=(1\!+\!x_1\!+\!x_2\!+\!x_3)\bI_3,\ 
	\bbeta=\begin{pmatrix} \sin x_1\\\sin x_2\\\sin x_3\end{pmatrix},\
	\sigma=\frac{4}{1\!+\!x_1\!+\!x_2\!+\!x_3},\ 
	u_{\mathrm{ex}} = \sin\big(\pi(x_1\!+\!x_2\!+\!x_3)\big).
	\end{align}
	Here $\bI_3$ is the $3\times3$ identity matrix.
	The right-hand side $f$ is constructed in order to manufacture the solution $u_{\mathrm{ex}}$ in \eqref{eq:ex2}.
	Dirichlet boundary conditions are imposed on the entire boundary of the domain matching the exact solution.
	We consider a sequence of tetrahedral meshes obtained by refinement of an unstructured quasi-uniform tetrahedral initial mesh.
	The penalization parameters are chosen as $\gamma =50p^2$ and $K_F=\kmin=1$.
	
	In \Cref{fig:numex2} we show the absolute errors of the quasi-Trefftz and the standard (full-polynomial space) DG methods, for the same polynomial degrees $p\in\{2,3,4\}$ and under mesh refinement. 
	We observe that the quasi-Trefftz DG method converges with the expected orders $h^{p+1}$ in the $L^2(\Omega)$ norm and $h^p$ in the $\vertiii{\cdot}_\mathrm{dar}$-norm, the latter in agreement with \Cref{th:finalerror} and both norms matching the convergence rates of the standard DG method.
	The errors of the two methods are similar, but the quasi-Trefftz DG error is slightly larger by a constant factor (within a factor 1.65 for the $\vertiii{\cdot}_\mathrm{dar}$-error for $h<0.5$).
	
	The assembly of the quasi-Trefftz DG linear system has an overhead given by the computation of the basis functions and the particular approximate solution $u_{h,f}$.
	To assess this, in \Cref{tab:ex2} we compare the computing time of the quasi-Trefftz and the full-polynomial version of the DG method. 
	We observe that, as soon as $h$ is sufficiently small or $p$ large, the quasi-Trefftz version requires considerably less time: the basis computation time is offset by the reduced number of degrees of freedom. 
	
	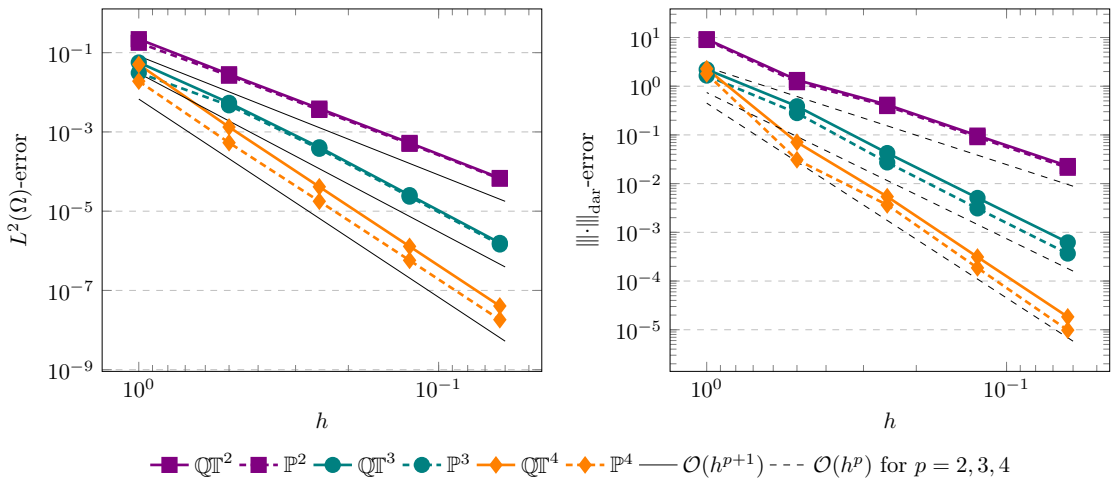
\begin{figure}[ht!]\centering
	\resizebox{.98\linewidth}{!}{
	\begin{tikzpicture}
	\begin{groupplot}[
	group style={
	group name={my plots},
	group size=2 by 1,
	horizontal sep=2cm,
	},
	legend style={
	legend columns=8,
	at={(0.8,-0.2)},
	draw=none
	},
	xmajorgrids=true,
	ymajorgrids=true,
	grid style=dashed,
	cycle list name=paulcolors,
	]      
	\nextgroupplot[ymode=log,xmode=log,x dir=reverse, ylabel={$L^2(\Omega)$-error},xlabel={$h$}]
	\foreach \k in {2,3,4}{
	\addplot+[discard if not={p}{\k},discard if not={qt}{1}] table [x=h, y=l2error, col sep=comma] {qtell3d3t.csv};
	\addplot+[discard if not={p}{\k},discard if not={qt}{0}] table [x=h, y=l2error, col sep=comma] {qtell3d3t.csv};
	}
	\addplot[domain=0.06:1.0] {exp(-3*ln(1/x)-2.5)};
	\addplot[domain=0.06:1.0] {exp(-4*ln(1/x)-3.5)};
	\addplot[domain=0.06:1.0] {exp(-5*ln(1/x)-5.0)};
	\nextgroupplot[ymode=log,xmode=log,x dir=reverse, ylabel={$\vertiii{\cdot}_\mathrm{dar}$-error},xlabel={$h$}]
	\foreach \k in {2,3,4}{
	\addplot+[discard if not={p}{\k},discard if not={qt}{1}] table [x=h, y=dgerror, col sep=comma] {qtell3d3t.csv};
	\addplot+[discard if not={p}{\k},discard if not={qt}{0}] table [x=h, y=dgerror, col sep=comma] {qtell3d3t.csv};
	}
	\addlegendimage{solid}
	\addplot[dashed,domain=0.06:1.0] {exp(-2*ln(1/x)+0.9)};
	\addplot[dashed,domain=0.06:1.0] {exp(-3*ln(1/x)-0.3)};
	\addplot[dashed,domain=0.06:1.0] {exp(-4*ln(1/x)-0.8)};
	\legend{$\QT^2$,$\IP^2$,$\QT^3$,$\IP^3$,$\QT^4$,$\IP^4$,$\mathcal O(h^{p+1})$,{$\mathcal O(h^p)$ for $p=2,3,4$},}
	\end{groupplot}
	\end{tikzpicture}}
    \vspace{-0.5em}
	\caption{
	Error norms for the non-homogeneous problem in the unit cube, with the right-hand side and coefficients chosen to manufacture the solution given in \eqref{eq:ex2}.
	We compare the quasi-Trefftz method ($\QT^p$) to the standard DG method using the full polynomial spaces ($\IP^p$) for 
	polynomial degrees $p=2,3,4$ on the same mesh sequence.
	Reference lines for the optimal convergence rates $\mathcal O(h^{p+1})$ and $\mathcal O(h^p)$ are shown in full and dashed lines, respectively.
	}
	\label{fig:numex2}
	\end{figure}
	\begin{table}[ht!]\centering
	\begin{filecontents*}{qtell3d3t.csv}
qt,hnr,h,p,l2error,dgerror,eoc,dgeoc,nel,localdofs,totaldofs,totaltime,setuptime,solvetime,cond
0,0,1.0,2,0.17749876543006635,8.890103626475103,0,0,12,10.0,120,0.005597591400146484,0.0034782886505126953,0.002117633819580078,0
0,1,0.5,2,0.026160672988375672,1.2083416149499266,2.7623374301364416,2.87917185351656,96,10.0,960,0.030763864517211914,0.00538182258605957,0.02537989616394043,0
0,2,0.25,2,0.0035843918083641505,0.39128961436287785,2.8675994034052588,1.6267196585022237,768,10.0,7680,0.30583810806274414,0.06396746635437012,0.24186420440673828,0
0,3,0.125,2,0.0004993352947510695,0.08990037897373004,2.843647556328951,2.121837716106226,6144,10.0,61440,5.719448089599609,0.19458627700805664,5.524856090545654,0
0,4,0.0625,2,6.57212601994362e-05,0.02114166400897238,2.925576835745519,2.088238265203809,49152,10.0,491520,208.3757610321045,1.6072964668273926,206.76845455169678,0
0,0,1.0,3,0.03090398046042414,1.6503194643417818,0,0,12,20.0,240,0.050765037536621094,0.02781200408935547,0.022950410842895508,0
0,1,0.5,3,0.004679708953978259,0.2802440977942781,2.723301958607347,2.5579894305356703,96,20.0,1920,0.1036534309387207,0.0289151668548584,0.07473564147949219,0
0,2,0.25,3,0.0003744121054946509,0.027177121647806213,3.643719820780343,3.366219416585439,768,20.0,15360,0.95688796043396,0.08212137222290039,0.8747622966766357,0
0,3,0.125,3,2.3219168640195594e-05,0.0030955283123081153,4.011238858198889,3.134135107505278,6144,20.0,122880,24.333662509918213,0.5770204067230225,23.75662899017334,0
0,4,0.0625,3,1.440359062692052e-06,0.0003659776219174247,4.010815910953432,3.0803583132836545,49152,20.0,983040,1064.5029900074005,3.7722129821777344,1060.7307636737823,0
0,0,1.0,4,0.01908264121254444,1.7979663748751238,0,0,12,35.0,420,0.13344025611877441,0.08939290046691895,0.04404330253601074,0
0,1,0.5,4,0.0005342319690164228,0.03085408437854749,5.158650746181144,5.8647607479467245,96,35.0,3360,0.22845458984375,0.0556788444519043,0.17277073860168457,0
0,2,0.25,4,1.777680278296628e-05,0.0036530356034037087,4.909398531865866,3.078293762255086,768,35.0,26880,2.796959161758423,0.1719679832458496,2.6249847412109375,0
0,3,0.125,4,5.746658510767579e-07,0.00018628695913512465,4.951128742716285,4.293497226386123,6144,35.0,215040,80.88720321655273,1.3403542041778564,79.54683876037598,0
0,4,0.0625,4,1.824084675001997e-08,9.814513462299545e-06,4.977478714203637,4.246466122585473,49152,35.0,1720320,4448.462619543076,9.319770574569702,4439.142832517624,0
1,0,1.0,2,0.21948526585693037,9.257895685948053,0,0,12,9.0,108,0.03509116172790527,0.03297066688537598,0.0021185874938964844,0
1,1,0.5,2,0.028606234073332012,1.34221059150097,2.9397226055435066,2.7860732603357325,96,9.0,864,0.14272451400756836,0.08919024467468262,0.05353093147277832,0
1,2,0.25,2,0.003885033611548965,0.41653703119809515,2.8803305982554197,1.6880943810564786,768,9.0,6912,0.6223759651184082,0.4611186981201172,0.16124653816223145,0
1,3,0.125,2,0.0005288336772241929,0.09691973189675106,2.8770411215373426,2.10358244147262,6144,9.0,55296,7.3878090381622314,3.2368905544281006,4.150904655456543,0
1,4,0.0625,2,6.885817718940925e-05,0.02300879485253036,2.9411141580129647,2.074604993443214,49152,9.0,442368,184.71256566047668,25.11117935180664,159.60137343406677,0
1,0,1.0,3,0.0566168358484234,2.21026853208588,0,0,12,16.0,192,0.1870889663696289,0.16895365715026855,0.01813220977783203,0
1,1,0.5,3,0.0054218000145120175,0.38815996987701773,3.3843873186080224,2.5094984088446926,96,16.0,1536,0.3055136203765869,0.2540860176086426,0.05142402648925781,0
1,2,0.25,3,0.0004100646858129761,0.042502279158733614,3.724848488429302,3.1910392315345257,768,16.0,12288,2.104132652282715,1.49904203414917,0.6050794124603271,0
1,3,0.125,3,2.546641626870081e-05,0.005093943898059104,4.009183648715691,3.0606852312142454,6144,16.0,98304,25.238791942596436,11.210553646087646,14.028226375579834,0
1,4,0.0625,3,1.5838565353869962e-06,0.000622037549208022,4.007082384198786,3.03370949445959,49152,16.0,786432,724.8352961540222,91.252610206604,633.5826685428619,0
1,0,1.0,4,0.050257302961385066,2.275671597276707,0,0,12,25.0,300,0.5849611759185791,0.5599896907806396,0.02496790885925293,0
1,1,0.5,4,0.0013290308740807446,0.07079562533477822,5.240886726953024,5.006488351978599,96,25.0,2400,1.029137134552002,0.9242663383483887,0.10486412048339844,0
1,2,0.25,4,4.123564439536093e-05,0.005403573930514441,5.010338856385137,3.711674386412365,768,25.0,19200,6.499725818634033,5.187083005905151,1.3126370906829834,0
1,3,0.125,4,1.2923294355438278e-06,0.0003112709322354883,4.99584616477107,4.117671159871435,6144,25.0,153600,80.69153428077698,42.6611213684082,38.03040027618408,0
1,4,0.0625,4,4.0487061205934065e-08,1.8394937625041956e-05,4.996369145498697,4.080790168847831,49152,25.0,1228800,2219.6832604408264,345.96191120147705,1873.7213258743286,0
\end{filecontents*}
\pgfplotstabletypeset[
	col sep=comma,
	fixed, fixed zerofill=true,
	every head row/.style={before row=\toprule,after row=\midrule},
	every last row/.style={after row=\bottomrule},
	create on use/nhnr/.style={create col/set list={1.0,$2^{-1}$,$2^{-2}$,$2^{-3}$,$2^{-4}$}},
	create on use/nel/.style={create col/set list={$12$,$96$,$768$,$6144$,$49152$}},
	columns={nhnr,nel,totaltime,totaltime,totaltime,totaltime,totaltime,totaltime},
    display columns/0/.style={column name={Meshsize},string type,discard if not={qt}{0},discard if not={p}{2}},
    display columns/1/.style={column name={\#elements},string type,discard if not={qt}{0},discard if not={p}{2}},
	display columns/2/.style={column name={$\QT^2$},discard if not={qt}{1},discard if not={p}{2},column type/.add={>{\columncolor[gray]{.8}}}{}},
	display columns/3/.style={column name={$\IP^2$},discard if not={qt}{0},discard if not={p}{2}},
	display columns/4/.style={column name={$\QT^3$},discard if not={qt}{1},discard if not={p}{3},column type/.add={>{\columncolor[gray]{.8}}}{}},
	display columns/5/.style={column name={$\IP^3$},discard if not={qt}{0},discard if not={p}{3}},
	display columns/6/.style={column name={$\QT^4$},discard if not={qt}{1},discard if not={p}{4},column type/.add={>{\columncolor[gray]{.8}}}{}},
	display columns/7/.style={column name={$\IP^4$},discard if not={qt}{0},discard if not={p}{4}},
	]
	{qtell3d3t.csv}
	\caption{Timings for the non-homogeneous problem described in \eqref{eq:ex2}.
	We compare the quasi-Trefftz method to the standard DG method using the full polynomial spaces; the corresponding errors are plotted in Figure~\ref{fig:numex2}.
	The timings are given in seconds and include the time for setting up the finite element spaces, the assembly, and solving the linear system.
	Mesh generation is excluded.
	}\label{tab:ex2}
	\end{table}
	
	The left panel in Figure \ref{fig:cond} shows how the advantage provided by the quasi-Trefftz approach improves with higher polynomial degrees $p$.
	This seems to confirm the $\exp(-bp^{1/(d-1)})$ behavior of Trefftz and quasi-Trefftz errors, as opposed to the 
	$\exp(-cp^{1/d})$ dependence for methods based on classical polynomial spaces, see \cite[sec.~3.1]{hiptmair2016survey} and \cite[sec.~5.1.3, 5.2.2]{GomezMoiola2024}.
	Note however that we are not aware of any rigorous quasi-Trefftz $p$-convergence result.

    \subsubsection{Conditioning}\label{sec:cond}
	
	We study the condition number for the quasi-Trefftz DG method and the standard DG method. 
	We consider a 2D Dirichlet problem in the unit square $\Omega=(0,1)^2$ with coefficients
	$\bK=(1+x_1+x_2)\bI_2$, $\bbeta=(1,0)^\top$, and $\sigma=\frac{3}{1+x_1+x_2}$, with $\bI_2$ the identity matrix in $\mathbb{R}^{2\times 2}$.
	The right panel of Figure \ref{fig:cond} shows the condition numbers of the matrices for the quasi-Trefftz DG and the standard DG methods.
    For the standard DG they asymptotically approach the rate 
    $\mathcal{O}(h^{-2})$ for all $p\in\IN$, in accordance with the theory \cite{castillo2002performance}, while for the quasi-Trefftz DG the condition numbers grow asymptotically less than $\mathcal{O}(h^{-0.5})$ for all $p\in\IN$.
	However, for increasing values of $p$, the condition number of the quasi-Trefftz DG appears to grow exponentially.
	This is to be expected since we initialize the quasi-Trefftz Cauchy data by monomials.
	The selection of Cauchy data that ensure better-conditioned quasi-Trefftz bases is currently under investigation.
	
	\begin{figure}[ht!]
	\pgfplotscreateplotcyclelist{pcolors}{
	magenta, every mark/.append style={solid,fill=magenta},mark=square*,very thick,mark size=3pt\\
	cyan,every mark/.append style={solid,fill=cyan},mark=*,very thick,mark size=3pt\\
	}
    \begin{center}
	\resizebox{\linewidth}{!}{
	\begin{tikzpicture}
	\begin{groupplot}[
	group style={
	group name={my plots},
	group size=2 by 1,
	horizontal sep=2cm,
	},
	xmajorgrids=true,
	ymajorgrids=true,
	grid style=dashed,
	]      
	\nextgroupplot[
    ymode=log,  ylabel={$\vertiii{\cdot}_\mathrm{dar}$-error},xlabel={$\sqrt{\#\mathrm{DOFs}}$},cycle list name=pcolors]
	\foreach \qt in {0,1}{
	\addplot+[discard if not={h}{0.1},discard if not={qt}{\qt}] table [x=dofs, y=dgerror, col sep=comma] {qtellppp3d3t.csv};
	}
	\addlegendentryexpanded{$\IP^p$};
	\addlegendentryexpanded{$\QT^p$};
	\nextgroupplot[ymin=10^1,ymax=10^9.5, ymode=log,xmode=log,ylabel={cond(A)},xlabel={$h$},cycle list name=paulcolorsCOND, x dir=reverse,legend pos=south east]
		\foreach \p in {3,4,5}{
			\addplot+[discard if not={p}{\p},discard if not={qt}{1}] table [x=h, y=cond, col sep=comma] {qtell2d3tcond.csv};
			\addplot+[discard if not={p}{\p},discard if not={qt}{1},discard if={eoc_cond}{0}, only marks,
			visualization depends on=\thisrow{eoc_cond} \as \labela,
			nodes near coords=\pgfmathprintnumber{\labela}
			,
			every node near coord/.append style={
				black,
				draw=yellow!30,
				ellipse,
				fill=yellow!30,
				inner sep=1pt,
				xshift=-4.5ex,
				yshift=-1ex,
				scale=0.5,/pgf/number format/fixed,
				/pgf/number format/precision=2,/pgf/number format/fixed zerofill}
			] table [x=h, y=cond, col sep=comma] {qtell2d3tcond.csv};
			\addplot+[discard if not={p}{\p},discard if not={qt}{0}] table [x=h, y=cond, col sep=comma] {qtell2d3tcond.csv};
			\addplot+[discard if not={p}{\p},discard if not={qt}{0},discard if={eoc_cond}{0}, only marks,
			visualization depends on=\thisrow{eoc_cond} \as \labela,
			nodes near coords=\pgfmathprintnumber{\labela}
			,
			every node near coord/.append style={
				black,
				draw=yellow!30,
				ellipse,
				fill=yellow!30,
				inner sep=1pt,
				xshift=-4.5ex,
				yshift=-1.5ex,
				scale=0.45,/pgf/number format/fixed,
				/pgf/number format/precision=2,/pgf/number format/fixed zerofill}
			] table [x=h, y=cond, col sep=comma] {qtell2d3tcond.csv};}
		\legend{$p=3$,,,,$p=4$,,,,$p=5$}
	\end{groupplot}    

	\end{tikzpicture}}
    \end{center}
    \vspace{-2em}
	\caption{Left: $p$-convergence comparison between quasi-Trefftz and full polynomials DG in terms of degrees of freedom and computational time for the problem with coefficients \eqref{ex2} using $h = 0.1$.
    Right: Condition numbers of the quasi-Trefftz DG (solid lines) and the standard DG (dashed lines) matrices for the Dirichlet problem on the unit square stated in \cref{sec:cond}; the numbers in the yellow markers show the algebraic rate in $h$ of the corresponding segment.
    }
	\label{fig:cond}
	\end{figure}
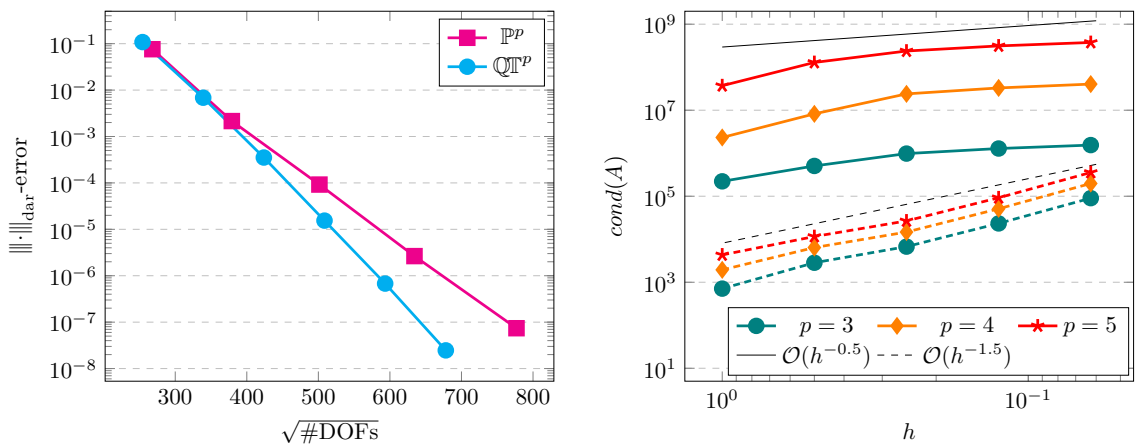
	
	\subsection{Advection-dominated problems}\label{ex3}
	
	We investigate advection-dominated problems to assess the capabilities of the method also in such more challenging setting.
	In section \ref{layer} we consider a solution that presents an internal layer while in section \ref{lshaped} a solution with boundary layers and corner singularities. 
	In both examples the advection field $\bbeta$ is divergence-free and the reaction $\sigma=0$, hence assumption \eqref{assumptionbeta} is violated.
	Even if the stability theory does not apply, the method performs well.
	
	\subsubsection{Internal layer}\label{layer}
	We consider the homogeneous problem $f\equiv 0$ with coefficients
	\begin{align*}
	\bK=\nu \bI_2,\qquad \bbeta=(x_2 e^{x_1-\frac12x_2^2},\;e^{x_1-\frac12x_2^2})^\top, \qquad \sigma=0,
	\end{align*}
	in $\Omega=(0,1)^2$. The streamlines of the divergence-free advection field $\bbeta$ are the parabolas $x_1=\frac{x_2^2}{2}+c$.
	We consider different values of the parameter $\nu$ to investigate the influence of the advection term: $\nu=10^{-j}$ for $j=1,2,3,4$. 
	We set the Dirichlet boundary $\Gamma_D=\{(x_1,x_2)\in\partial \Omega\mid x_1=0 \text{ or } x_2=0\}$ with the data $g_D=1$ if $x_1\le1/3$ and $g_D=0$ otherwise,
	and Neumann boundary $\Gamma_N=\partial \Omega\setminus \Gamma_D$ with the data $g_N=0$.
	The choice of the penalization parameter for this kind of problems is particularly delicate, as for small values of $\gamma$ coercivity fails, but large values of $\gamma$ often introduce spurious oscillations in the solution.
	Here we choose $\gamma = 100$ and $K_F=\kmin=\nu$.
	
	The results with mesh size $h=2^{-6}$ and $p=3$ are shown in \Cref{fig:numex3}, where we compare the results of the quasi-Trefftz method (lower row) to those of the full-polynomial DG (upper row).
	Both methods show similar results: the flat part of the solution is well approximated, and the discontinuity at the boundary and the internal layer are well captured with small oscillations.
	
	In Figure \ref{fig:nu} we investigate numerically the dependence of the error on the diffusion parameter $\nu$, recall Remark~\ref{rem::Peclet}.
	We show the $L^2(\Omega)$-norm of the error for $\nu=10^{-1},\dots,10^{-5}$.
	For each method, we compare the solutions obtained with $h=2^{-5}$ to the solutions obtained with a fine mesh ($h=2^{-7}$) with both methods. We see no significant difference between the two methods even for small values of $\nu$, indeed for both methods we observe a growth of order 
	$\mathcal{O}(\nu^{-0.2})$: the quasi-Trefftz space does not spoil the robustness of the DG scheme in the advection-dominated regime.
    We observe the same behavior also for the analogous problem with a small positive reaction coefficient $\sigma$ (numerical results not reported here), thus satisfying the assumptions of Theorem~\ref{th:errorestimate}.
	
	\begin{figure}[htbp]
	\vspace{-3mm}
	\includegraphics[width=0.24\linewidth]{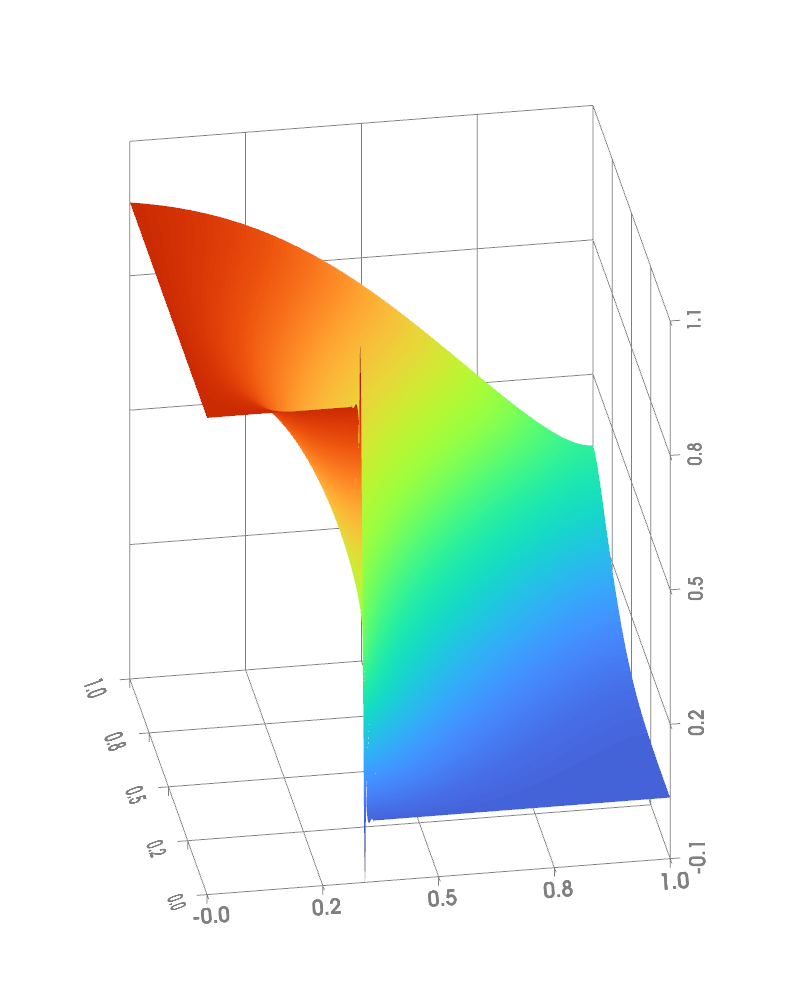}
	\includegraphics[width=0.24\linewidth]{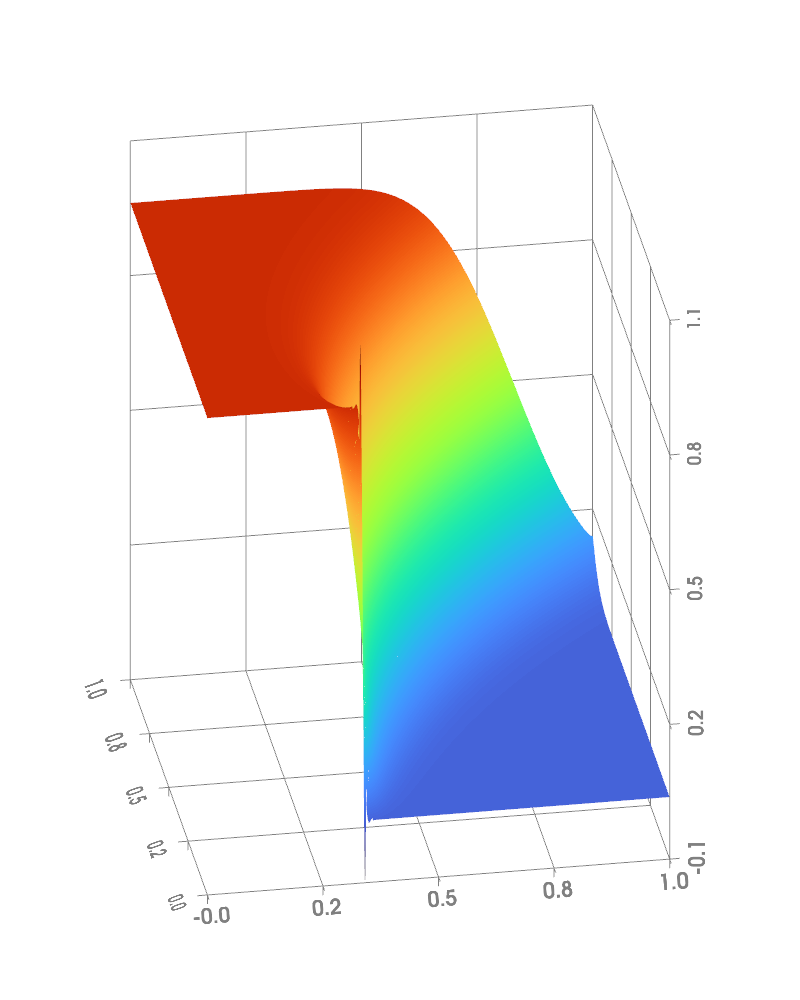}
	\includegraphics[width=0.24\linewidth]{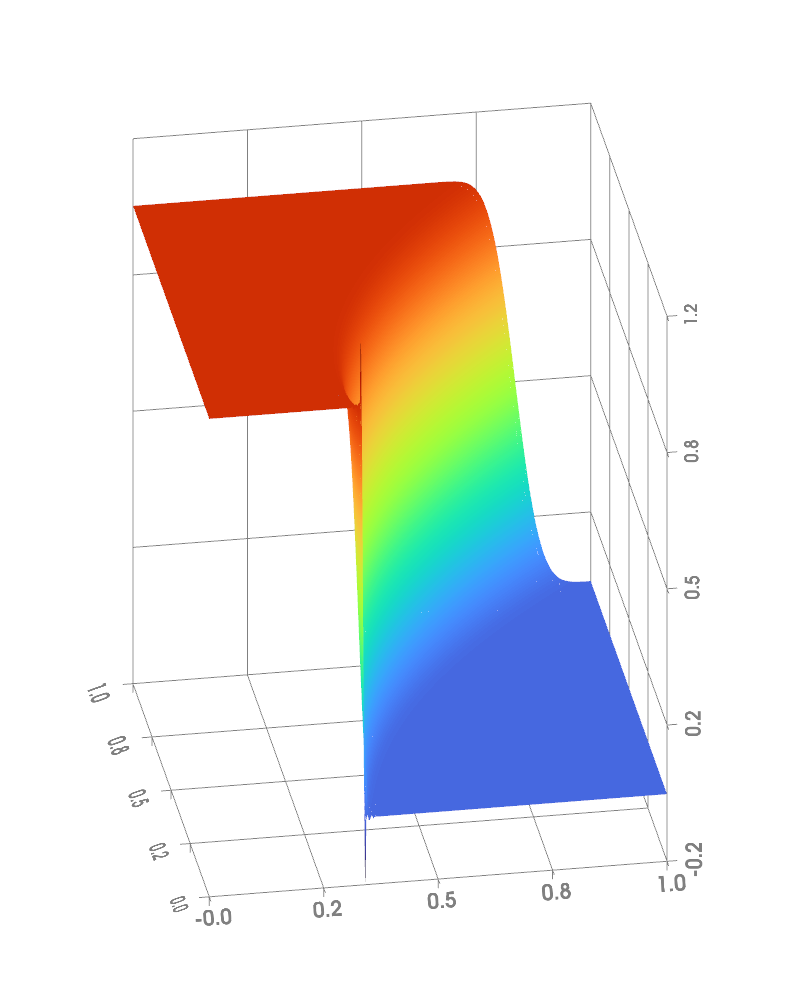}
	\includegraphics[width=0.24\linewidth]{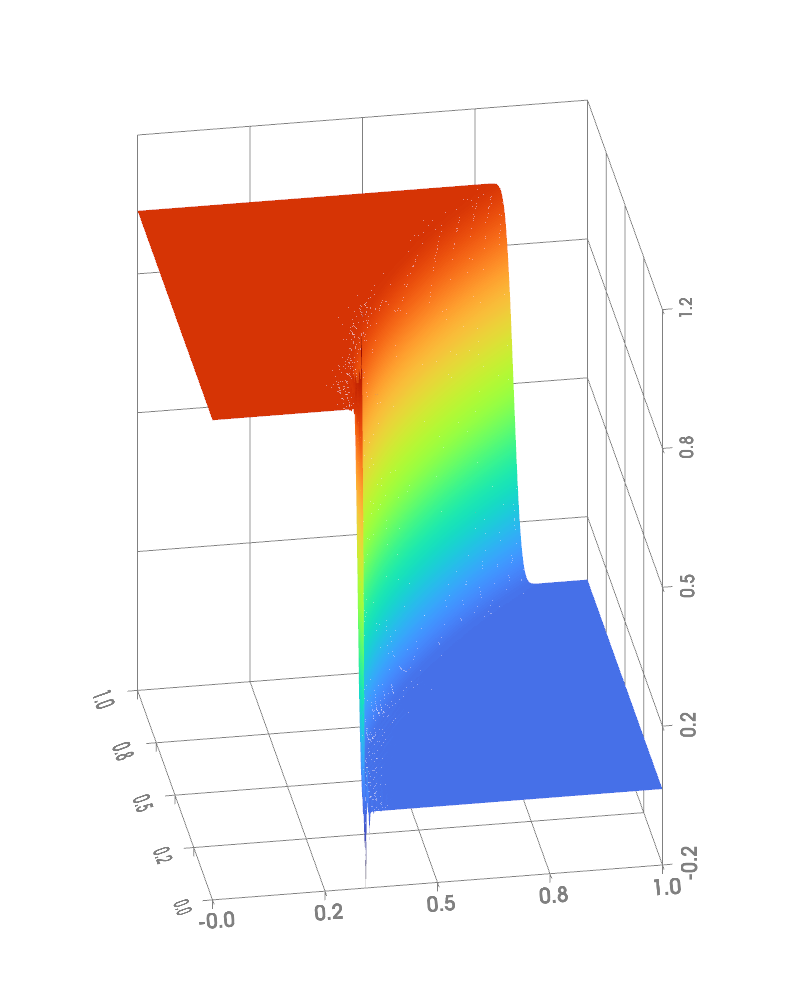}
	\vspace{-3mm}
	\includegraphics[width=0.24\linewidth]{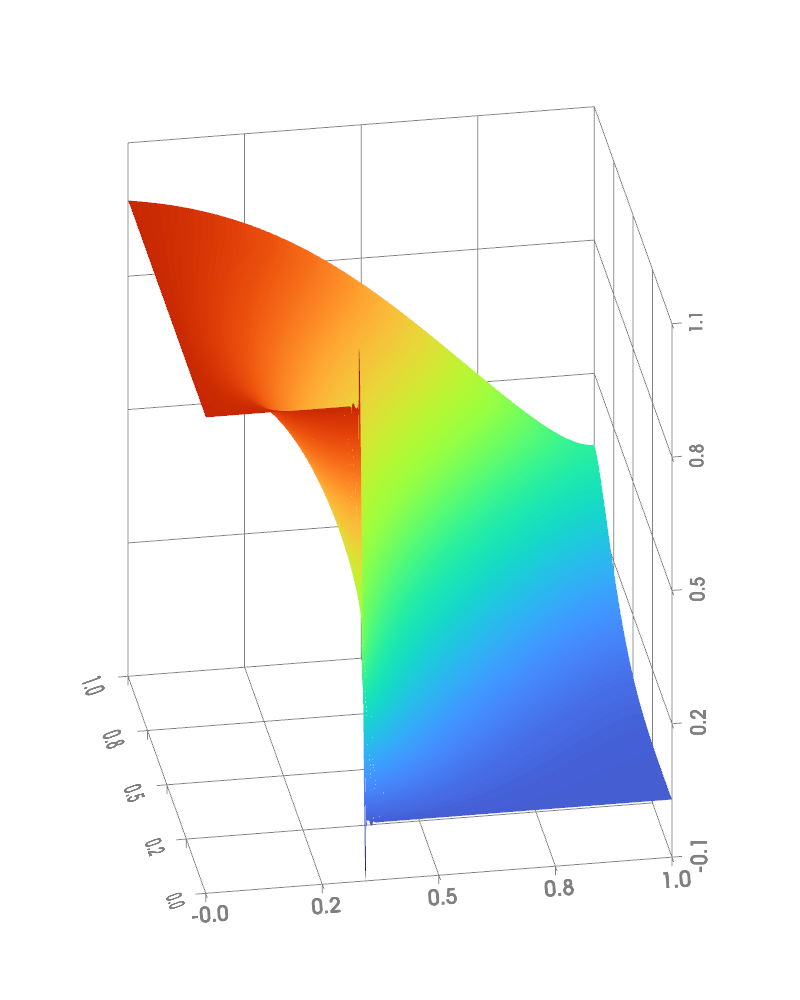}
	\includegraphics[width=0.24\linewidth]{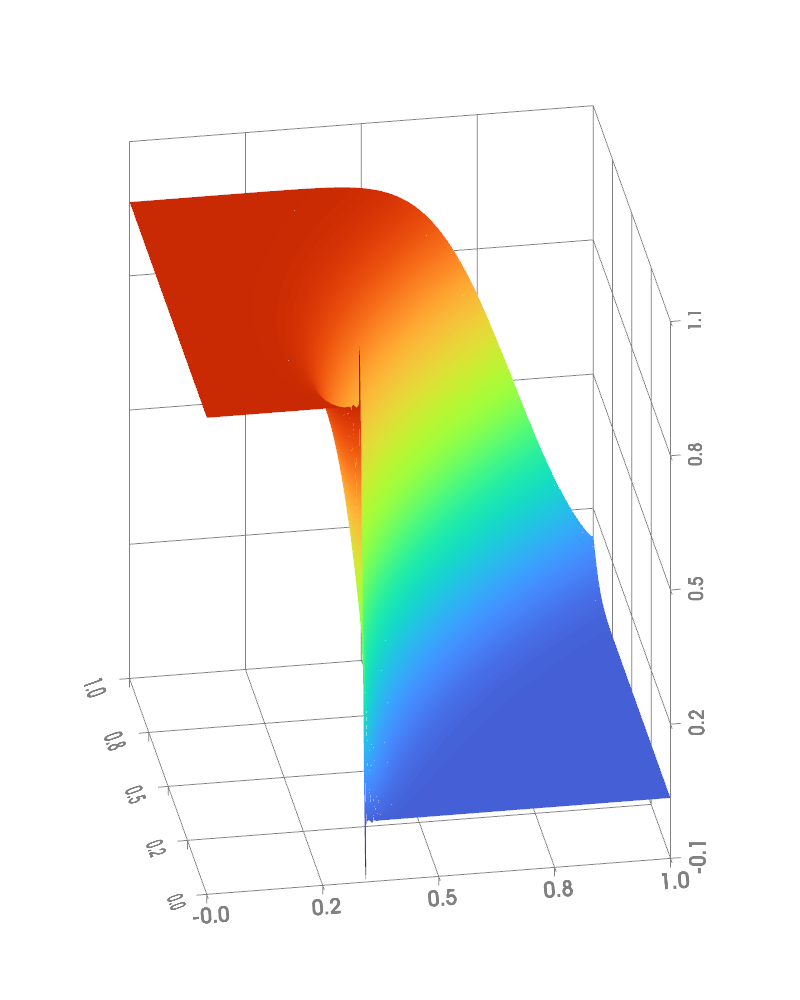}
	\includegraphics[width=0.24\linewidth]{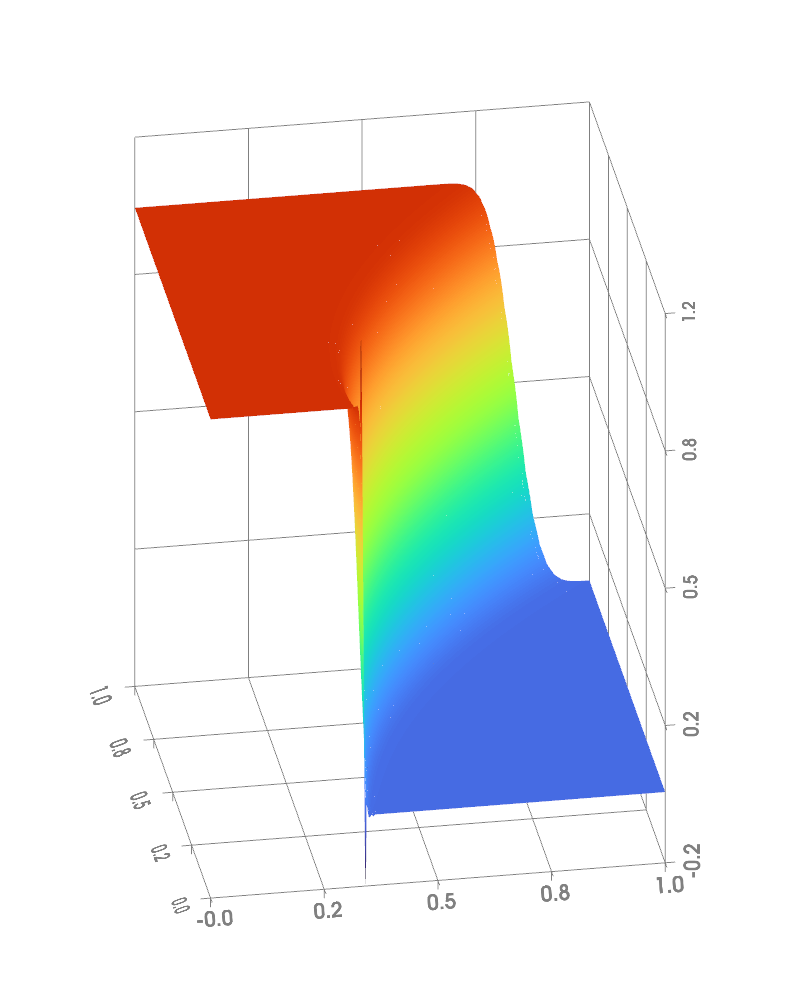}
	\includegraphics[width=0.24\linewidth]{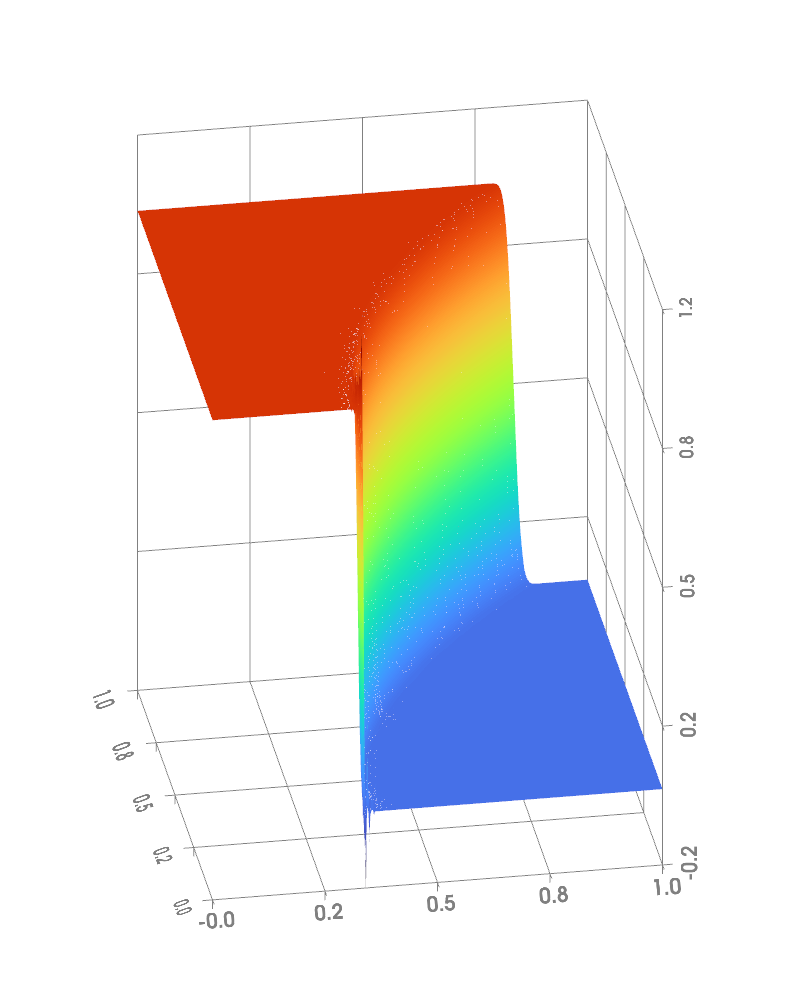}
	\caption{Numerical result for the advection-dominated problem of section \ref{layer}.
	The first row shows results for the full polynomial space and the the second for the quasi-Trefftz space.
	From the first to the last column we vary the diffusion coefficient $\nu=10^{-j}$ for $j=1,2,3,4$.}
	\label{fig:numex3}
	\end{figure}
	
	\begin{figure}[htbp]
	\pgfplotscreateplotcyclelist{pcolors}{
		magenta, every mark/.append style={solid,fill=magenta},mark=square,very thick,mark size=5pt\\
		cyan,every mark/.append style={solid,fill=cyan},mark=o,very thick,mark size=5pt\\
		orange,every mark/.append style={solid,fill=orange},mark=diamond,very thick,mark size=5pt\\
		violet, every mark/.append style={solid,fill=violet},mark=star,very thick,mark size=5pt\\
	}
	\begin{center}
		{
			\begin{tikzpicture}[scale=.85]
				\begin{groupplot}[
					group style={
						group name={my plots},
						group size=1 by 1,
						horizontal sep=2cm,
					},
					legend style={
						legend columns=1,
						},
					xmajorgrids=true,
					ymajorgrids=true,
					grid style=dashed,
					]      
					\nextgroupplot[
					ymode=log, 	xmode=log,  ylabel={$\N{\cdot}_{L^2(\Omega)}$-error},xlabel={$\nu$},cycle list name=pcolors,
					legend pos=outer north east,
					]
                    \foreach \qt in {0,1}{
						\addplot+[discard if not={qt}{\qt}] table [x=nu, y=error, col sep=comma] {Peclet.csv};
					}
					\foreach \qt in {1,0}{
					\addplot+[discard if not={qt}{\qt}] table [x=nu, y=error_cross, col sep=comma] {Peclet.csv};
				}
				\foreach \qt in {0}{
				\addplot+[discard if not={qt}{\qt},discard if={eoc}{nan}, only marks,
				visualization depends on=\thisrow{eoc} \as \labela,
				nodes near coords=\pgfmathprintnumber{\labela}
				,
				every node near coord/.append style={
					black,
					draw=yellow!30,
					ellipse,
					fill=yellow!30,
					inner sep=1pt,
					xshift=5ex,
					yshift=-2.5ex,
					scale=0.7,/pgf/number format/fixed,
					/pgf/number format/precision=2,/pgf/number format/fixed zerofill}
				] table [x=nu, y=error, col sep=comma] 
			{Peclet.csv};}			
\addlegendentryexpanded{$\N{u_{2^{-5},\IP}-u_{2^{-7},\IP}}_{L^2(\Omega)}$};
\addlegendentryexpanded{$\N{u_{2^{-5},\QT}-u_{2^{-7},\QT}}_{L^2(\Omega)}$};
\addlegendentryexpanded{$\N{u_{2^{-5},\IP}-u_{2^{-7},\QT}}_{L^2(\Omega)}$};
\addlegendentryexpanded{$\N{u_{2^{-5},\QT}-u_{2^{-7},\IP}}_{L^2(\Omega)}$};
	\end{groupplot}    
	\end{tikzpicture}}
	\end{center}
    \vspace{-1.5em}
	\caption{Error dependence on the diffusion parameter $\nu$ for the advection-dominated problem of section \ref{layer}. The numbers in the yellow boxes are the empirical rates.}
	\label{fig:nu}
\end{figure}
	
\subsubsection{L-shaped domain}\label{lshaped}
	We apply the method to a strongly advection-dominated BVP from \cite[sec.~4]{brezzi1998applications}. 
	The coefficients are 
	\begin{align}\label{eq:exL}
	\bK=\nu \bI_2,\qquad \bbeta=(-x_2,x_1)^\top,\qquad \sigma=0,
	\end{align}
	with $\nu=5\times 10^{-3}$.
	The source term is $f=0$, the problem is posed on the L-shaped domain $\Omega=(0,1)^2\setminus[0,0.5]^2$,
	and the Dirichlet boundary condition $g_D=1$ on $x_2=0$ and $g_D=0$ elsewhere is imposed on $\partial\Omega$.
	The solution $u$ exhibits boundary layers and corner singularities.
	
	We fix $K_F=\kmin=\nu$, $\gamma = 50$ and choose a mesh 4246 triangular elements and polynomial degree $p=3$.
	\Cref{fig:numex3} shows the quasi-Trefftz DG solution, in perfect visual agreement with \cite[Fig.~12]{brezzi1998applications}, and the difference against the full-polynomial space DG solution.
	We observe that this difference is concentrated at the singular corners and at the outflow layer.
	
	\begin{figure}[htbp]
	\vspace{-1mm}
    \begin{center}
	\includegraphics[width=0.49\linewidth,clip,trim=140 120 30 120]{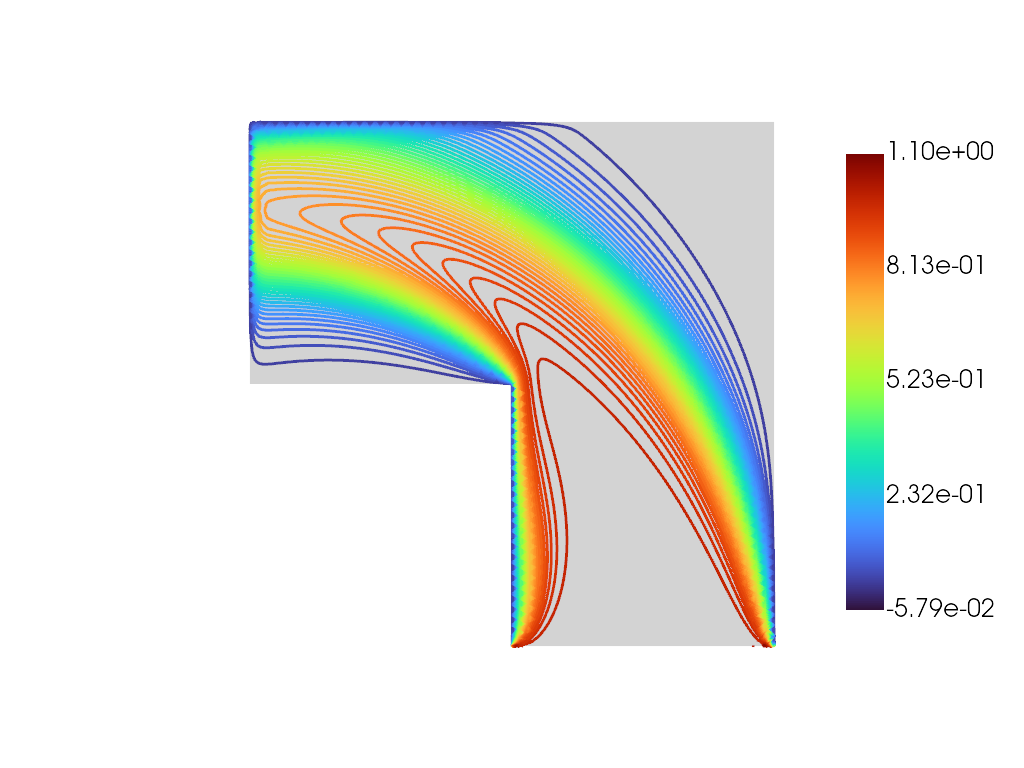}
	\includegraphics[width=0.49\linewidth,clip,trim=140 120 30 120]{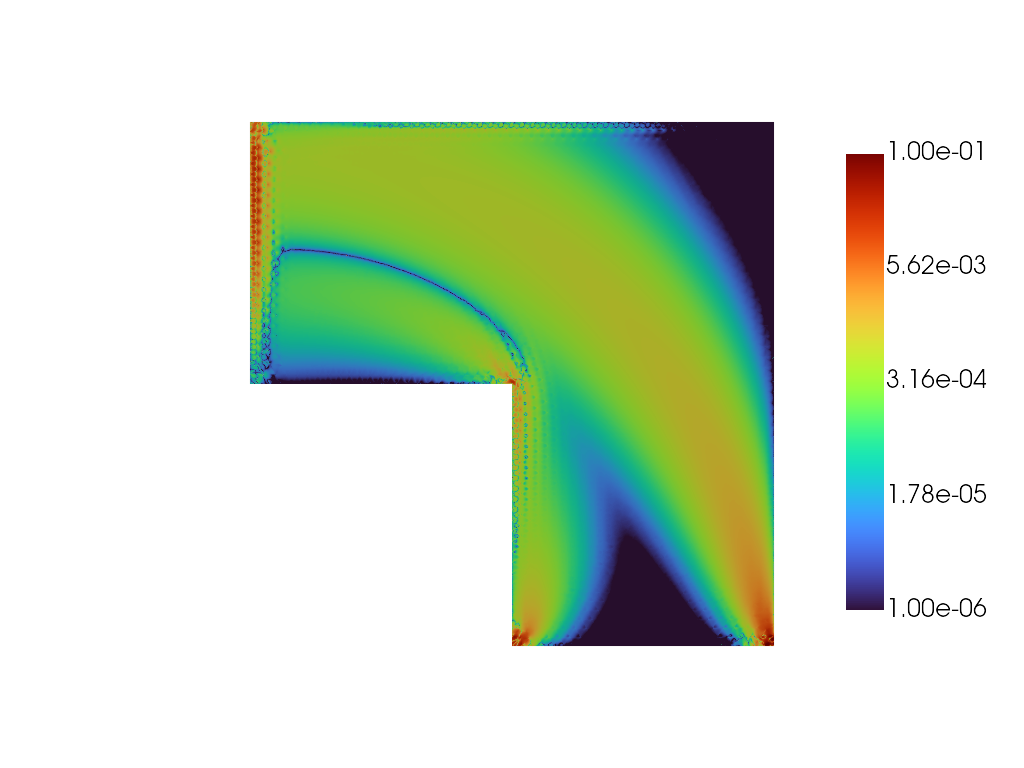}
    \end{center}
\vspace{-1em}
	\caption{Numerical result for the Dirichlet problem \eqref{eq:exL} on the L-shaped domain, computed using $h=0.02$, $p=3$ and $\gamma=50$.
	Left: contour plot of the quasi-Trefftz DG discrete solution, linear color scale (cf.\ \cite[Fig.~12]{brezzi1998applications}). 
	Right: difference between the solutions of the full-polynomial and the quasi-Trefftz DG scheme on the same mesh, logarithmic color scale truncated at $10^{-6}$.}
	\label{fig:Lshaped}
	\end{figure}
	
	\section{Conclusions and future developments} 
	We have examined the polynomial quasi-Trefftz space for linear PDEs with smooth coefficients and right-hand side. We have shown it can approximate smooth solutions to the PDE with the same accuracy as the full polynomial space but requiring fewer degrees of freedom, and described an algorithm for the construction of a quasi-Trefftz basis.
	Then we have analyzed a quasi-Trefftz DG method for elliptic diffusion--advection--reaction BVPs with piecewise-smooth data,
	proving optimal-rate $h$-convergence, as confirmed by numerical results.
	
	Further investigation into the properties of the quasi-Trefftz basis functions is needed to optimize the choice of the Cauchy data, i.e.\ of the $m$ polynomial bases in the initialization step of the algorithm for the construction of the quasi-Trefftz basis functions, aiming at further improving accuracy, conditioning and computing time.
	
	Analysis of non-polynomial quasi-Trefftz functions could be useful for efficiently approximating solutions with boundary layers or less regular solutions, such as those with corner singularities.
	
	Further research is required to obtain approximation estimates in Sobolev norms and to establish optimal DG error bounds in $L^2(\Omega)$-norm, as suggested by the numerics. 
	A challenging extension, which has not yet been achieved for quasi-Trefftz methods, is the analysis of the approximation properties for increasing polynomial degrees (\textit{p}-convergence).
	
	Another interesting extension is the application of this method to PDEs whose nature changes in the domain, such as the Euler-Tricomi equation ($\partial^2_{x}u +x \partial^2_{y}u= 0$), modeling transonic flow. 	
	
	\section*{Acknowledgements}
	LMIG, AM and PS gratefully acknowledge the Centro Internazionale per la Ricerca Matematica (CIRM, Trento) for hosting them in the Research-in-Pairs program.
	AM and CP acknowledge support from PRIN projects ``ASTICE'' (202292JW3F) and ``NA-FROM-PDEs'' (201752HKH8), GNCS--INDAM, and PNRR-M4C2-I1.4-NC-HPC-Spoke6, which are partly funded by the European Union -- NextGenerationEU.
This research was funded in part by the Austrian Science Fund (FWF) 
\href{https://doi.org/10.55776/F65}{10.55776/F65} and 
\href{https://doi.org/10.55776/ESP4389824}{10.55776/ESP4389824}.
For open access purposes, the authors have applied a CC BY public copyright license to any author-accepted manuscript version arising from this submission.
	LMIG acknowledges support from the US National Science Foundation (NSF): this material is based upon work supported by the NSF under Grant No. \href{https://www.nsf.gov/awardsearch/showAward?AWD_ID=2110407&HistoricalAwards=false}{DMS-2110407}.
	
	\printbibliography[heading=bibintoc,title={Bibliography}]
	
	\appendix
	\section{Estimates on jump--average terms}\label{s:Preliminaryestimates}
	\begin{lemma}\label{lemma:Appendix1}
	For all $v,w\in H^2(\calT_h)$, 
	\begin{equation}\label{consistencytermbound}
	\Bigg|\sum_{F\in\calF_h^{\mathrm I}\cup \calF_h^{\mathrm D}} \int_F  \mvl{\bK\nabla v} \cdot \jmp{ w}\Bigg| \leq
	\N{\bK}_{L^{\infty}(\Omega)}^{\frac12}
	\Bigg(\sum_{E\in\calT_h} \sum_{F\in\calF_E} \frac{h_F}{\gamma K_F} 
	\N{(\bK^{\frac12}\nabla v)_{|_E}\cdot \bn_F}_{L^2(F)}^2
	\Bigg)^{\frac12} \abs{w}_{\mathrm J}.
	\end{equation}
	\end{lemma}
	\begin{proof}
	Using the Cauchy--Schwarz inequality and recalling the definition \eqref{eq:Norms} of $|\cdot|_{\mathrm J}$, we have
	\begin{equation*}
	\begin{split}
	\Bigg|\sum_{F\in\calF_h^{\mathrm I}\cup \calF_h^{\mathrm D}} \int_F  \mvl{\bK\nabla v} \cdot \jmp{ w}\Bigg|
	\leq &
	\Bigg(\sum_{F\in\calF_h^{\mathrm I}\cup \calF_h^{\mathrm D}}  \frac{h_F}{\gamma K_F}
	\int_F \left(\mvl{\bK\nabla v}\cdot \bn_F\right)^{2}
	\Bigg)^{\frac12}
	|w|_{\mathrm J}.
	\end{split}
	\end{equation*}
	In the first term, for all $F=\partial E_1 \cap \partial E_2\in\calF_h^{\mathrm I}$, Young's inequality yields
	\begin{align*}
	\int_F \left( \mvl{\bK   \nabla v} \cdot \bn_F\right)^2&
	=\int_F\left[\frac12 
	\left(\bK_{|_{E_1}}^{\frac12} (\bK^{\frac12}\nabla v)_{|_{E_1}}+\bK_{|_{E_2}}^{\frac12} (\bK^{\frac12}\nabla v)_{|_{E_2}}\right)  \cdot \bn_F\right]^2\\
	&	\leq
	\frac12 	\N{\bK}_{L^{\infty}(\Omega)} 
	\bigg(\N{(\bK^{\frac12}\nabla v)_{|_{E_1}} \cdot \bn_F}^2_{L^2(F)}+\N{(\bK^{\frac12}\nabla v)_{|_{E_2}}\cdot \bn_F}^2_{L^2(F)}\bigg).
	\end{align*}
	For all $F\in\calF_h^{\mathrm D}$ with $F\subset\partial E$, we obtain
	$
	\int_F \left( \mvl{\bK   \nabla v} \cdot \bn_F\right)^2
	\leq
	\N{\bK}_{L^{\infty}(\Omega)} \|(\bK^{\frac12}\nabla v)_{|_E} \cdot \bn_F\|^2_{L^2(F)}.
	$
	Combining these two bounds, the thesis is derived by collecting the facet contributions of each mesh element.
	\end{proof}

	\begin{lemma}\label{lemma:Appendix2}
	For all $(v,w_h)\in H^2(\calT_h) \times  V_h$,
	\begin{equation}\label{consistencyterm}
	\Bigg|\sum_{F\in\calF_h^{\mathrm I}\cup \calF_h^{\mathrm D}} \int_F  \jmp{v}\cdot \mvl{\bK   \nabla w_h} \Bigg| 
	\leq
	\abs{v}_{\mathrm J}
	\frac{\N{\bK}_{L^{\infty}(\Omega)}}{\kmin }
	\bigg( \frac{N_{\partial} (p+1)(p+d)}{\gamma\,r_\star}\bigg)^{\frac12}
	\Bigg(\sum_{E\in\calT_h}\N{\bK^{\frac12}\nabla w_h}_{L^2( E)}^2\Bigg)^{\frac12}.
	\end{equation}
	\end{lemma}
	\begin{proof}
	As in the previous proof, Cauchy--Schwarz inequality leads to
	\begin{equation*}
	\begin{split}
	\Bigg|\sum_{F\in\calF_h^{\mathrm I}\cup \calF_h^{\mathrm D}} \int_F  \jmp{v}\cdot \mvl{\bK   \nabla w_h}\Bigg|
	\leq 
	\Bigg(\sum_{F\in\calF_h^{\mathrm I}\cup \calF_h^{\mathrm D}}  \frac{h_F}{\gamma K_F}
	\int_F \left(\mvl{\bK\nabla w_h}\cdot \bn_F\right)^{2}
	\Bigg)^{\frac12}
	|v|_{\mathrm J}.
	\end{split}
	\end{equation*}
	For all $F\in\calF_h^{\mathrm I}$ with $F=\partial E_1 \cap \partial E_2$, Young's inequality yields
	\begin{align*}
	\int_F \left( \mvl{\bK   \nabla w_h} \cdot \bn_F\right)^2&
	=\int_F\left[\frac12 
	\left(	(\bK\nabla w_h)_{|_{E_1}}+ (\bK\nabla w_h)_{|_{E_2}}\right)  \cdot \bn_F\right]^2\\
	&	\leq
	\frac12 	\N{\bK}_{L^{\infty}(\Omega)}^2 \left(\N{(\nabla w_h)_{|_{E_1}} \cdot \bn_F}^2_{L^2(F)}+\N{(\nabla w_h)_{|_{E_2}}\cdot \bn_F}^2_{L^2(F)}\right),
	\end{align*}
	and for $F\in\calF_h^{\mathrm D}$ with $F\subset\partial E$, we have 
	$\int_F \left( \mvl{\bK   \nabla w_h} \cdot \bn_F\right)^2
	\leq \N{\bK}_{L^{\infty}(\Omega)}^2 \N{(\nabla w_h)_{|_E} \cdot \bn_F}^2_{L^2(F)}$.
	Aggregating the contributions of each element, owing to the facts that $h_F\leq h_E$ for all $F\in\calF_E$, $E\in\calT_h$, that $\kmin \leq K_F$ for all $F\in\calF_h$, and to the discrete trace inequality \eqref{discretetraceinequality}, we deduce 
	\begin{align}\label{Koutside}
	&\Bigg|\sum_{F\in\calF_h^{\mathrm I}\cup \calF_h^{\mathrm D}} \int_F  \jmp{v}\cdot \mvl{\bK\nabla w_h} \Bigg| 
	\leq \abs{v}_{\mathrm J}
	\N{\bK}_{L^{\infty}(\Omega)}
	\bigg(\sum_{E\in\calT_h} \sum_{F\in\calF_E} \frac{h_F}{\gamma K_F} 
	\N{(\nabla w_h)_{|_E}\cdot \bn_F}_{L^2(F)}^2
	\bigg)^{\frac12}\\
	& \leq\abs{v}_{\mathrm J}
	\N{\bK}_{L^{\infty}(\Omega)} 
	\bigg( 	\sum_{E\in\calT_h} \sum_{F\in\calF_E} \frac{h_E}{\gamma \kmin }\frac{(p+1)(p+d)}{r_\star}  h_E^{-1}
	\N{\nabla w_h}_{L^2(E)}^2
	\bigg)^{\frac12} .\nonumber
	\end{align}
	The assertion is obtained recalling the definition \eqref{Npartial} of $N_{\partial}$ and using the ellipticity condition \eqref{ellipticity} as $\|\nabla w_h\|\le \kmin ^{-\frac12}\|\bK^\frac12\nabla w_h\|$.
	\end{proof}
	
	In the proof of \Cref{lemma:Appendix2} we could have used the bound \eqref{consistencytermbound} where $\bK$ is already inside the $L^2$-norm, instead of using \eqref{Koutside} and paying the factor $\frac{1}{\kmin }$. 
	However, in general $\bK^{\frac12}\nabla v_h$ is not a polynomial, so the classical discrete trace inequality \eqref{discretetraceinequality} would not be applicable.
\end{document}